\documentclass[letterpaper,10pt]{article}  

\usepackage{cite}
\usepackage{amsmath}
\usepackage{amssymb,amsfonts}
\usepackage{algorithmic}
\usepackage{graphicx}
\usepackage{textcomp}

\usepackage{widetext}

\usepackage{lipsum} 
\usepackage{amsthm} 
\usepackage{float}
\usepackage{url}

\usepackage[switch]{lineno}  
\usepackage[export]{adjustbox} 

\usepackage{nccmath} 
\usepackage{xfrac} 
\usepackage{dsfont} 
\usepackage[margin=1in]{geometry} 

\usepackage{algorithm,algorithmic}

\usepackage{graphicx} 
\usepackage{subcaption} 

\usepackage{makecell} 
\usepackage{multirow} 
\usepackage{tabu} 
\newcommand\numberthis{\addtocounter{equation}{1}\tag{\theequation}} 

\newcommand{\ts}{\textsuperscript} 
\usepackage{soul} 

\newcommand{\Ebb}{\mathbb{E}} 
 
\newcommand{\Rbb}{\mathbb{R}} 
\newcommand{\Ecal}{\mathcal{E}} 
\newcommand{\Gcal}{\mathcal{G}} 
\newcommand{\Lcal}{\mathcal{L}} 
\newcommand{\Ncal}{\mathcal{N}} 
\newcommand{\Scal}{\mathcal{S}} 
\newcommand{\Ucal}{\mathcal{U}} 
\newcommand{\Vcal}{\mathcal{V}} 

\newcommand{\ones}{\mathbf{1}} 
\newcommand{\zeros}{\mathbf{0}} 
\newcommand{\Exp}[1]{\Ebb{\left[#1\right]}} 

\newcommand{\diag}{\text{diag}}

\newcommand{\range}{\text{range}}

\newcommand\norm[1]{\left\lVert#1\right\rVert} 
\newcommand\normtwo[1]{\left\lVert#1\right\rVert_2} 
\newcommand\abs[1]{\left|#1\right|} 
\newcommand\paren[1]{\left(#1\right)} 
\newcommand\Sbraces[1]{\left[#1\right]} 
\newcommand\Cbraces[1]{\left\lbrace#1\right\rbrace} 
\newcommand\Abraces[1]{\left\langle#1\right\rangle} 

\DeclareMathOperator*{\argmax}{argmax} 

\DeclareMathOperator*{\maximize}{maximize}
\renewcommand{\st}{\text{subject to}}

\newcommand{\ram}{\rightarrow} 

\newcommand{\algindent}{\hspace{\algorithmicindent}} 
\newcommand{\MYSTATE}{\vspace{0.2em}\STATE} 

\newtheoremstyle{mytheoremstyle} 
    {\topsep}                    
    {\topsep}                    
    {\normalfont}                   
    {}                           
    {\bfseries}                   
    {.}                          
    {.5em}                       
    {}              
\theoremstyle{mytheoremstyle}

\newtheorem{theorem}{Theorem}

\newtheorem{definition}{Definition}
\newtheorem{lemma}{Lemma}
\newtheorem{corollary}[theorem]{Corollary}
\newtheorem{assumption}{Assumption}
\newtheorem{remark}{Remark}
\newtheorem{fact}{Fact}

\renewcommand{\AA}{{A^+ A}} 
\newcommand{\LL}{{\Lambda^+ \Lambda}} 
\newcommand\normAA[1]{\left\lVert#1\right\rVert_A} 
\newcommand{\SM}{{\text{SM}}} 
\newcommand\normSM[1]{\left\lVert#1\right\rVert_{\SM}} 
\newcommand{\SMNO}{{\text{SMNO}}} 
\newcommand\normSMNO[1]{\left\lVert#1\right\rVert_{\SMNO}} 
\newcommand{\sign}{\text{sign}}

\newcommand{\SML}{{\text{SML}}} 
\newcommand\normSML[1]{\left\lVert#1\right\rVert_{\SML}}

\begin{document}

\title{Setwise Coordinate Descent for Dual Asynchronous Decentralized Optimization\footnote{This article has been accepted for publication in IEEE Transactions on Automatic Control, DOI 10.1109/TAC.2025.3543463. \\
© 2025 IEEE.  Personal use of this material is permitted.  Permission from IEEE must be obtained for all other uses, in any current or future media, including reprinting/republishing this material for advertising or promotional purposes, creating new collective works, for resale or redistribution to servers or lists, or reuse of any copyrighted component of this work in other~works.}
}

\author{Marina Costantini$^{1,\dagger}$ \and Nikolaos Liakopoulos$^2$ \and Panayotis Mertikopoulos$^3$ \and Thrasyvoulos Spyropoulos$^{1,4}$}
\date{\normalsize{
    $^1$EURECOM, Sophia Antipolis\\%
    $^2$Amazon, Luxembourg\\
    $^3$Univ. Grenoble Alpes, CNRS, Inria, Grenoble INP, LIG\\
    $^4$Technical University of Crete\\
    $^\dagger$Correspondence: marina.costantini@eurecom.fr } }

\maketitle

%
%

\begin{abstract}

In decentralized optimization over networks, synchronizing the updates of all nodes incurs significant communication overhead. For this reason, much of the recent literature has focused on the analysis and design of asynchronous optimization algorithms where nodes can activate anytime and contact a single neighbor to complete an iteration together. However, most works assume that the neighbor selection is done at random based on a fixed probability distribution (e.g., uniform), a choice that ignores the optimization landscape at the moment of activation. Instead, in this work we introduce an optimization-aware selection rule that chooses the neighbor providing the highest dual cost improvement (a quantity related to a dualization of the problem based on consensus). This scheme is related to the coordinate descent (CD) method with the Gauss-Southwell (GS) rule for coordinate updates; in our setting however, only a subset of coordinates is accessible at each iteration (because each node can communicate only with its neighbors), so the existing literature on GS methods does not apply. To overcome this difficulty, we develop a new analytical framework for smooth and strongly convex functions that covers our new class of \textit{setwise CD algorithms} –a class that applies to both decentralized and parallel distributed computing scenarios– and we show that the proposed setwise GS rule can speed up the convergence in terms of iterations by a factor equal to the size of the largest coordinate set. We analyze extensions of these algorithms that exploit the knowledge of the smoothness constants when available and otherwise propose an algorithm to estimate these constants. 
Finally, we validate our theoretical results through extensive simulations.

\end{abstract}

%
%

\section{Introduction} \label{sec:introduction}

Many timely applications require solving optimization problems over a network where nodes can only communicate with their direct neighbors.
This may be due to the need of distributing storage and computation loads (e.g. training large machine learning models \cite{lian2017can}), or to avoid transferring data that is naturally collected in a decentralized manner, either due to the communication costs or to privacy reasons (e.g. sensor networks \cite{wan2009event}, edge computing \cite{alrowaily2018secure}, and precision medicine \cite{warnat2021swarm}). 

Specifically, we consider a setting where the nodes want to solve the decentralized optimization problem
\begin{equation} \label{eq:primal_unconstrained}
    \underset{\theta \in \Rbb^d}{\text{minimize}} \quad \sum_{i=1}^n f_i(\theta),
\end{equation}

where each local function $f_i$ is known only by node $i$ and nodes can exchange optimization values (parameters, gradients) but \emph{not} the local functions themselves.
We represent the communication network as a graph $\Gcal = (\Vcal,\Ecal)$ with $n \, {=} \, |\Vcal| \, {<} \, \infty$ nodes (agents) and $E \, {=} \, |\Ecal|$ edges, which are the links used by the nodes to communicate with their neighbors. 

Existing methods to solve this problem usually assign to each node a local variable $\theta_i$ and execute updates that interleave a local gradient step at the nodes followed by an averaging step (usually carried out by a doubly-stochastic matrix) that aggregates the update of the node with those of its neighbors \cite{nedic2007rate, neglia2020decentralized, jakovetic2018convergence, lian2018asynchronous, yuan2016convergence, shi2015extra}.

Alternatively, the consensus constraint can be stated explicitly between node pairs connected by an edge:   
\begin{subequations} \label{eq:primal_constrained}
    \begin{alignat}{2} 
    & \underset{\theta_1,\ldots,\theta_n \in \Rbb^d}{\text{minimize}} \quad && \sum_{i=1}^n f_i(\theta_i) \label{eq:problem-a} \\ 
    & \st \quad && \theta_i = \theta_j \quad \forall \; (i,j) \equiv \ell \in \mathcal{E}, \label{eq:problem-b}
    \end{alignat} 
\end{subequations}

where $\ell \equiv (i,j)$ indicates that edge $\ell$ links nodes $i$~and~$j$.
This reformulation allows solving the decentralized optimization problem through its dual problem. 
Interestingly, dual decentralized algorithms are among the very few that have been shown to achieve optimal convergence rates \cite{scaman2017optimal,hendrikx2019accelerated,uribe2020dual} when using acceleration.
However, these algorithms are \textit{synchronous}, in the sense that they require coordinating the updates of all nodes at each iteration. 
In contrast, by exploiting connections with coordinate descent theory naturally derived from the dual formulation, here we propose an alternative way to speed up convergence in dual decentralized optimization that does \textit{not} require network-wide synchronicity.  

In our setup, nodes activate anytime at random and select one of their neighbors to make an update together.
Methods with such minimal coordination requirements avoid incurring extra costs of synchronization that may also slow down convergence
\cite{iutzeler2013asynchronous, wei2013convergence, xu2017convergence, srivastava2011distributed, ram2009asynchronous}.
However, most of these works assume that when a node activates, it simply selects the neighbor to contact randomly, based on a predefined probability distribution. 
This approach overlooks the possibility of letting nodes \emph{choose} the neighbor to contact taking into account the optimization landscape at the time of activation.
Therefore, here we depart from the probabilistic choice and ask:
\emph{can nodes pick the neighbor smartly to make the optimization converge faster?}

In this paper, we give an affirmative answer and propose algorithms that achieve this by solving the dual problem of \eqref{eq:primal_constrained}. 
In the dual formulation, there is one dual variable $\lambda_\ell \in \Rbb^d$ per constraint $\theta_i = \theta_j$, hence each dual variable can be associated with an edge $\ell$ in the graph. 
Our algorithms let an activated node $i$ contact a neighbor $j$ so that together they update their shared variable $\lambda_\ell, \; \ell \equiv (i,j)$ with a gradient step.
In particular, we propose to select the neighbor $j$ such that the updated $\lambda_\ell$ is the one \emph{whose directional gradient for the dual function is the largest}, and thus the one that provides the greatest cost improvement at that iteration.~Beyond the preliminary version of this work that appeared in \cite{costantini2022pick}, 
such dual greedy choice has not yet been considered in the literature.

Interestingly, the above protocol where a node activates and selects a $\lambda_\ell$ to update can be seen as applying the coordinate descent (CD) method \cite{nesterov2012efficiency} to solve the dual problem of \eqref{eq:primal_constrained}, with the following key difference: unlike standard CD methods, where \emph{any} of the coordinates may be updated, now \emph{only a small subset of coordinates are accessible at each step}, which are the coordinates associated with the edges connected to the node activated.
Moreover, our proposal of updating the $\lambda_\ell$ with the largest gradient is similar to the Gauss-Southwell (GS) rule\cite{nutini2015coordinate}, but applied \emph{only} to the parameters accessible by the activated node.

We name such protocols \emph{setwise CD} algorithms, and we analyze a number of possibilities for the coordinate (or equivalently, neighbor) choice: random sampling, local GS, and two extensions that take into account the smoothness constants of the dual function. 
Compared to standard CD literature, three difficulties complicate the analysis and constitute the base of our contributions:
(i) for arbitrary graphs, the dual problem of \eqref{eq:primal_constrained} has an objective function that is \emph{not} strongly convex, even if the primal functions $f_i$ \emph{are} strongly convex,
(ii) the fact that the GS rule is applied to a few coordinates prevents the use of standard norms to obtain the linear rate, as commonly done for CD methods \cite{nesterov2012efficiency, nutini2015coordinate, nutini2017let}, and 
(iii) the coordinate sets are overlapping (i.e. non-disjoint), which makes the problem even harder.

Our results also apply to the (primal) parallel distributed setting where multiple workers modify different sets of coordinates of the parameter vector, which is stored in a server accessible by all workers \cite{tsitsiklis1986distributed, peng2016arock, xiao2019dscovr}.
In particular, we show that for this setting the GS selection attains the maximum speedup promised by the theory (see Theorem \ref{theo:rate_SGS-CD}).

Our contributions can be summarized as follows: 

\begin{itemize}
    \item We introduce the class of \textit{setwise CD} algorithms for asynchronous optimization applicable to both the decentralized and the parallel distributed settings. 
    
    \item We prove linear convergence rates for smooth and strongly convex $f_i$ and four coordinate (equivalently, neighbor) choices: random uniform,  GS, and their variants when the coordinate smoothness constants are known.

    \item For the cases when these constants are not known, we propose an algorithm based on backtracking \cite{nesterov2012efficiency} for estimating them online, and show that for certain problems this method achieves faster per-iteration convergence than the exact knowledge of these constants. 

    \item To obtain the rates of all considered algorithms, we prove strong convexity in unique norms derived from our dual coordinate updates. For the algorithms using the GS rule in particular, we deviate significantly from standard proofs and we resort to carefully-crafted support norms for each algorithm that allow us to bypass the difficulties introduced by the setwise choice.

    \item We show that the speedup in terms of number of iterations of GS selection with respect to random uniform can be up to $N_{\max}$ (the size of the largest coordinate set). 

    \item We prove that the versions accounting for coordinate smoothness are provably faster than those that do not account for these constants. In particular, we show that the algorithm exploiting both the GS rule \textit{and} the smoothness knowledge is provably faster than all others.

    \item We support all our results with thorough simulations\footnote{The code to reproduce the results is available at \url{https://github.com/m-costantini/Set-wise_Coordinate_Descent/}.}.
\end{itemize}

%
%

\section{Related work} \label{sec:related_work}

A number of algorithms have been proposed to solve \eqref{eq:primal_unconstrained} asynchronously, where we consider the edge-asynchronous model in Assumption \ref{ass:asynchronous_model} (other models have been considered in the literature, e.g. \cite{bastianello2020asynchronous, tian2020achieving}). These algorithms require transmitting a fixed number of vectors (either 1, 2 or 4 in the references below) over a single edge to complete an iteration.
In \cite{ram2009asynchronous}, the activated node chooses a neighbor uniformly at random and both nodes average their primal local values.
In \cite{iutzeler2013asynchronous} the authors adapted the ADMM algorithm to the decentralized setting,
but it was the ADMM of \cite{wei2013convergence} the first one shown to converge at the same rate as the centralized ADMM. 
The algorithm of \cite{xu2017convergence} tracks the average gradients to converge to the exact optimum instead of just a neighborhood around it, as many algorithms back then. 
The algorithm of \cite{pu2020push} can be used on top of directed graphs, which impose additional challenges. 
In \cite{mao2020walkman, hendrikx2023principled} a token is passed from node to node following a Markov chain and the node who receives the token completes an iteration.
A key novelty of our scheme, compared to the works above, is that we consider the possibility of letting the nodes \emph{choose smartly} the neighbor to contact in order to make convergence faster.

Work \cite{verma2023maximal} is, to the best of our knowledge, the only work similarly considering smart neighbor selection. 
The authors propose Max-gossip, a version of the (primal) algorithm in \cite{nedic2007rate} where the activated node averages its local parameter with that of the neighbor with whom the parameter difference is the largest. 
They show that the algorithm converges sublineraly to the optimum  (for convex functions), and show in numerical simulations that it outperforms random neighbor selection.
In contrast, here we propose dual algorithms for which we show linear convergence rates (for smooth and strongly convex functions), and most importantly, (i) we \textit{prove analytically} that either applying the GS rule and/or using the Lipschitz information achieves faster convergence than random neighbor sampling, and (ii) we \textit{quantify} the magnitude of the gains. 

To achieve this, we apply coordinate descent in the dual, similar to a few related works. 
In \cite{notarnicola2016asynchronous} the authors propose a dual proximal algorithm for non-smooth functions whose edge-based variant, if applied to smooth functions, reduces to the SU-CD algorithm presented here. 
The ESDACD algorithm of \cite{hendrikx2019accelerated} applies accelerated coordinate descent to the dual problem, achieving faster convergence at the expense of sacrificing complete asynchronicity (nodes need to agree on a sequence of node activations beforehand). 
Still, none of these works considered smart, non-randomized neighbor selection, 

Finally, as mentioned earlier, our work relates to standard CD literature. 
In particular, our theorems extend the results in \cite{nutini2015coordinate}, where the GS rule was shown to be up to $d$ times faster than uniform sampling for $f{:}\ \Rbb^d {\ram} \Rbb$, to the case where this choice is constrained to a subset of the coordinates only, sets have different sizes, each coordinate belongs to exactly two sets, and sets activate uniformly at random. 
As we explain in Sections \ref{sec:algorithms} and \ref{sec:algos_with_Lips}, these considerations bring important new challenges with respect to the standard single-machine CD algorithms.
Furthermore, our algorithms are not only applicable to the decentralized case but also to parallel distributed settings such as \cite{tsitsiklis1986distributed, peng2016arock, xiao2019dscovr}.
For the latter, \cite{you2016asynchronous} also analyzed the GS applied to coordinate subsets, but their sets are disjoint, accessible by any worker, and they do not quantify the speedup of the method with respect to random uniform sampling.

%
%

\section{Dual formulation} \label{sec:model_and_problems}

Here we define the notation, obtain the dual problem of \eqref{eq:primal_constrained}, and analyze the properties of the dual objective function.  

\begin{assumption}[Smoothness and strong convexity]
The $f_i$ are $M_i$-smooth and $\mu_i$-strongly convex, i.e. there exist finite constants $M_i \geq \mu_i > 0, i \in [n]$ such that:

\begin{align*}
    f_i(y) &\leq f_i(x) + \Abraces{\nabla f_i(x), y-x} + (M_i / 2) \normtwo{y-x}^2  \\ 
    f_i(y) &\geq f_i(x) + \Abraces{\nabla f_i(x), y-x} + (\mu_i / 2) \normtwo{y-x}^2. 
\end{align*} 
\end{assumption}

We define the concatenated primal and dual variables 
$\theta = [\theta_1^T,\ldots,\theta_n^T]^T \in \Rbb^{nd}$ and 
$\lambda = [\lambda_1^T,\ldots,\lambda_E^T]^T \in~\Rbb^{Ed}$, respectively.  
The graph's incidence matrix $A \in \Rbb^{n \times E} $ has exactly one 1 and one -1 per column $\ell$, in the rows corresponding to nodes $i,j: \ell \equiv (i,j)$, and zeros elsewhere (the choice of sign for each node is irrelevant).
We call $u_i \in \Rbb^n$ the vector that has 1 in entry $i$ and 0 elsewhere; 
we define $e_\ell \in \Rbb^E$ analogously.  
We use $k \in [K]$ to indicate $k=1,\ldots,K$. 
Vectors $\ones$ and $\zeros$ are respectively the all-one and all-zero vectors, and $I_d$ is the $d \times d$ identity matrix. 
Finally, in order to use matrix operations in the equations below for some operations, we define the block arrays $\Lambda = A \otimes I_d \in \Rbb^{nd \times Ed}$ and 
$U_i = u_i \otimes I_d \in \Rbb^{nd \times d}$, where $\otimes$ is the Kronecker product.
This operation generates arrays analogous to $A$ and $u_i$ where the original entries 1, -1, and 0 have been replaced by $I_d$, $-I_d$, and the all-zero $d \times d$ matrix, respectively.

We can rewrite now \eqref{eq:problem-b} as $\Lambda^T \theta = \zeros$, and the node variables as $\theta_i = U_i^T \theta$.
The minimum value of \eqref{eq:primal_constrained} satisfies:
\begin{align*} 
\underset{\theta: \Lambda^T \theta = \zeros}{\inf}  & \sum_{i=1}^n f_i(U_i^T \theta)
\stackrel{(\text{a})}{=} \inf_{\theta} \sup_{\lambda} \Sbraces{ \sum_{i=1}^n f_i (U_i^T \theta) - \lambda^T \Lambda^T \theta } \\
&\stackrel{(\text{b})}{=} \sup_{\lambda} \inf_{\theta} \Sbraces{ \sum_{i=1}^n f_i (U_i^T \theta) - \lambda^T \Lambda^T \theta } \\
&= -\inf_{\lambda} \sup_{\theta} \sum_{i=1}^n \Sbraces{ (U_i^T \Lambda \lambda)^T U_i^T \theta -  f_i (U_i^T \theta) } \\
&= -\inf_{\lambda} \sum_{i=1}^n f_i^* (U_i^T \Lambda \lambda) 
\triangleq  -\inf_{\lambda} F(\lambda), \numberthis \label{eq:dual_problem}
\end{align*}

\noindent
where (a) holds due to Lagrange duality and (b) holds by strong duality (see e.g. Section 5.4 in \cite{boyd2004convex}).
Functions $f_i^*$ are the Fenchel conjugates of the $f_i$, and are defined as 

\[ f_i^*(y) = \sup_{x \in \Rbb^d} \paren{ y^T x - f_i(x) }. \]  

Our setwise CD algorithms converge to the optimal solution of \eqref{eq:primal_constrained} by solving \eqref{eq:dual_problem}. 
In particular, they update a single dual variable 
$\lambda_\ell,\ell \in [E]$ 
at each iteration and converge to some minimum value $\lambda^*$ of $F(\lambda)$. 

Since $\sum_{i=1}^n f_i(U_i^T \theta)$ in \eqref{eq:problem-a} is $M_{\max}$-smooth and $\mu_{\min}$-strongly convex in $\theta$, with $M_{\max} = \max_i M_i$ and $\mu_{\min} = \min_i \mu_i$, function $F$ is $L$-smooth with $L = \frac{\gamma_{\max}}{\mu_{\min}}$, where $\gamma_{\max}$ is the largest eigenvalue of $\LL$ (Section 4 in \cite{uribe2020dual}).
We call $\gamma^+_{\min}$ the smallest non-zero eigenvalue\footnote{The ``+" stresses that $\gamma^+_{\min}$ is the smallest \emph{strictly positive} eigenvalue.} of $\LL$. 

However, as shown next, function $F$ is \emph{not} strongly convex in the standard L2 norm, which is the property that usually facilitates obtaining linear rates in optimization  literature. 

\begin{lemma} \label{lemma:no_SC}
    $F$ is not strongly convex in $\normtwo{\cdot}$. 
\end{lemma}
\begin{proof}
    Since $\Lambda$ does not have full column rank in the general case (i.e., unless the graph is a tree), there exist $w~\in~\Rbb^{Ed}$ such that $w \neq \zeros$ and $F(\lambda) = F(\lambda + tw) \; \forall t \in \Rbb$.
\end{proof} 

Nevertheless, we can still show linear rates for the setwise CD algorithms using the following result.

\begin{lemma}[\textbf{\textit{Appendix C of} \cite{hendrikx2019accelerated}}] \label{lemma:SC_in_AA}
    $F$ is $\sigma_A$-strongly convex in the semi-norm $\normAA{x} \triangleq (x^T \Lambda^+ \Lambda x)^\frac{1}{2}$, with $\sigma_A = \frac{\gamma^+_{\min}}{M_{\max}}$. 
\end{lemma}
        
Above, $\Lambda^+$ denotes the pseudo-inverse of $\Lambda$. 
A key fact for the proofs in the next section is that matrix $\Lambda^+ \Lambda$ is a projector onto $\range(\Lambda^T)$, the column space of $\Lambda^T$.

To simplify the notation, in what follows we assume that $d=1$, so that $\Lambda = A$, $U_i = u_i$, and the gradient
$\nabla_\ell F(\lambda) = \frac{\partial F(\lambda)}{\partial \lambda_\ell}$ 
of $F(\lambda)$ in the direction of $\lambda_\ell$ is a scalar. 
In Section \ref{sec:d_larger_1} we discuss how to adapt our proofs to the case $d > 1$.

%
%

\section{Setwise Coordinate Descent Algorithms} \label{sec:algorithms}

In this section, we present the \emph{setwise CD} algorithms, which can solve generic convex problems 
(and \eqref{eq:dual_problem} in particular) optimally and asynchronously.
For solving \eqref{eq:dual_problem} we rely~on:
\begin{assumption}[Asynchronous model] \label{ass:asynchronous_model}
Nodes activate independently at random (potentially with different frequencies) and choose a neighbor with whom to complete an update. 
Multiple \emph{pairs} of nodes may complete their updates in parallel, 
provided that each node updates with a single neighbor.
\end{assumption}

Here we analyze two possibilities for the coordinate choice within the accessible coordinate subset: (i) sampling uniformly at random (SU-CD), and (ii) applying the GS rule (SGS-CD).
Their extensions when the coordinate-wise Lipschitz constants are known (or can be estimated) are analyzed in Section \ref{sec:algos_with_Lips}.

If coordinate $\ell$ is updated at iteration $k$ and assuming \hbox{$d=1$}, the standard CD update applied to $F(\lambda)$~is \cite{nesterov2012efficiency}:
\begin{equation} \label{eq:CD_update}
    \lambda^{k+1} = \lambda^k - \eta^k \nabla_\ell F(\lambda^k) e_\ell,
\end{equation}
where $\eta^k$ is the stepsize. Since $F$ is $L$-smooth, choosing $\eta^k~{=}~1/L \; \forall k$ guarantees descent at each iteration \cite{nutini2015coordinate}:
\begin{equation} \label{eq:iter_progress}
    F(\lambda^{k+1}) \leq F(\lambda^k) - \frac{1}{2L} \paren{\nabla_\ell F(\lambda^k)}^2.
\end{equation}
  
Eq. \eqref{eq:iter_progress} will be the departure point to prove the linear convergence rates of SU-CD and SGS-CD. 
We remark that since $F$ is smooth and convex, the fixed stepsize choice above leads to exact convergence to an optimum of the function \cite{nesterov2012efficiency}.

We now define formally the setwise CD algorithms. 

\begin{definition}[\textit{\textbf{Setwise CD algorithm}}] \label{def:setwise_CD}
    In a setwise CD algorithm, every coordinate 
    $\ell \in [E]$ is assigned to (potentially multiple) sets 
    $\Scal_i,i \in [n]$, such that all coordinates belong to at least one set.
    At any 
    time, a set $\Scal_i$ may activate with uniform probability among the $i$; 
    a setwise CD algorithm then chooses a single coordinate $\ell \in \Scal_i$ to update using \eqref{eq:CD_update}.
\end{definition}

The next remark shows how the decentralized problem \eqref{eq:primal_constrained} can be solved asynchronously with setwise CD~algorithms.

\begin{remark} \label{rem:setwise_for_asych_optim}
    By letting (i) the $E$ coordinates\footnote{If $d \; {>} \;1$, the standard CD terminology calls each $\lambda_\ell$ a ``block coordinate", i.e. a vector of $d$ coordinates out of the $E \cdot d$ of $F: \Rbb^{E \cdot d} \ram \Rbb$.} 
    in Definition \ref{def:setwise_CD} correspond to the dual variables $\lambda_\ell, \ell \in [E]$, and (ii) the $\Scal_i, i \in [n]$ be the sets of dual variables corresponding to the edges that are connected to each node $i$, nodes can run a setwise CD algorithm to solve \eqref{eq:dual_problem} (and thus, also \eqref{eq:primal_constrained}) asynchronously.
\end{remark}

Furthermore, the setwise CD algorithms can also be used in the parallel distributed setting, as explained below. 

\begin{remark} 
    In a parallel distributed setting, where the parameter vector $x \in \Rbb^E$ is stored\footnote{Also in this case we may actually have $x \in \Rbb^{d \cdot E}$ and each worker may update a block coordinate $x_\ell \in \Rbb^d$ at each iteration.} in a shared server and can be modified by $n$ workers, the (primal) optimization can be done with setwise CD algorithms by letting each worker $i \in [n]$ modify a subset $\Scal_i$ of coordinates such that when the worker activates it chooses one coordinate $\ell \in \Scal_i$ to modify, and all coordinates $\ell \in [E]$ belong to at least one set $\Scal_i, i \in [n]$. 
\end{remark}

We also note that, while in the parallel distributed setting each coordinate may belong to any number of sets in $[1,n]$, our results apply to the case where all coordinates belong to \textit{exactly} two sets (i.e. they can be modified by two workers), since our analysis is driven by the decentralized~setting.

In light of Remark \ref{rem:setwise_for_asych_optim}, in the following we illustrate the steps that should be performed by the nodes to run the setwise CD algorithms to find an optimal value $\lambda^*$. 
We first note that the gradient of $F$ in the direction\footnote{This is equivalent to saying ``the $\ell$-th (block) entry of the gradient $\nabla F$".} of $\lambda_\ell$ for $\ell \equiv (i,j)$ is 
\begin{equation} \label{eq:coord_gradient}
    \nabla_\ell F(\lambda) = A_{i\ell} \nabla f_i^*(u_i^T A \lambda) 
    + A_{j\ell} \nabla f_j^*(u_j^T A \lambda).
\end{equation}

The nodes can use \eqref{eq:CD_update} and \eqref{eq:coord_gradient} to update the $\lambda_\ell$ of the edges they are connected to as follows: 
each node $i$ keeps in memory the current values of $\lambda_\ell, \ell \in \Scal_i$, which are needed to compute $\nabla f_i^*(u_i^T A \lambda)$.
Then, when edge $\ell \equiv (i,j)$ is updated (either because node $i$ activated and contacted $j$, or vice versa), both $i$ and $j$ compute their respective terms in the right-hand side of \eqref{eq:coord_gradient} and exchange them through their link. 
Finally, both nodes compute \eqref{eq:coord_gradient} and update their copy of $\lambda_\ell$ applying \eqref{eq:CD_update}. 

Algorithms \ref{alg:SU-CD} and \ref{alg:SGS-CD} below detail these steps for SU-CD and SGS-CD, respectively.
We have used $\Ncal_i$ to indicate the set of neighbors of node $i$ (note that $\Scal_i = \{ \ell: \ell \equiv (i,j), j \in \Ncal_i \}$).
Table \ref{tab:set_deifnitions} shows this and other set-related notation we will use.

We now proceed to describe the SU-CD and SGS-CD algorithms in detail, and prove their linear convergence rates. 
We remark that the analyses that follow are given in terms of number of iterations. See Section \ref{sec:asynchrony_realistic} for a discussion considering communication and computation complexity.

%
%
\begin{table}[t]
\centering
    \caption{Set-related definitions}
    \label{tab:set_deifnitions}
\begin{tabu} {|ll|}  \hline
    $\Scal_i$       & Set of edges connected to node $i$ \\  
    $\Ncal_i$       & Set of neighbors of node $i$   \\ 
    $N_i$           & Degree of node $i$, i.e. $N_i = \abs{\Scal_i} = \abs{\Ncal_i}$ \\ 
    $N_{\max}$      & Maximum degree in the network, i.e. $\max_i N_i$ \\ 
    $T_i$           & Selector matrix of set $\Scal_i$ (see Definition 
    \ref{def:selector_matrices}) \\
    $\Scal_i'$      & Subset $\Scal_i' \subseteq \Scal_i$ such that 
                    $\Scal_i' \cap \Scal_j' = \emptyset$ if $i \neq j$ \\ 
    $T_i'$          & Selector matrix of set $\Scal_i'$ \\ 
    \rule{0pt}{10pt}$\overline{\Scal_i'}$   
                    & Complement set of $\Scal_i'$ such that 
                    $\overline{\Scal_i'} = \Scal_i \setminus \Scal_i'$ \\ 
    \rule{0pt}{10pt}$\overline{T_i'}$       
                    & Selector matrix of set $\overline{\Scal_i'}$ \vspace{0.2em} \\ 
    \hline
\end{tabu}
\end{table}

%
%

\subsection{Setwise Uniform CD (SU-CD)}

In SU-CD, the activated node chooses the neighbor uniformly at random, as shown in Algorithm \ref{alg:SU-CD}. 
We can compute the per-iteration progress of SU-CD taking expectation in \eqref{eq:iter_progress}:
\begin{align*} 
    \Ebb \Big[ F(\lambda^{k+1}) \mid \lambda^k & \Big]
    \leq F(\lambda^k) - \frac{1}{2L} \Exp{\paren{\nabla_\ell F(\lambda^k)}^2 \mid \lambda^k } \\
    &= F(\lambda^k) - \frac{1}{2Ln} \sum_{i=1}^n \frac{1}{N_i} \sum_{\ell \in \Scal_i}  
    \paren{\nabla_\ell F(\lambda^k)}^2 \\
    &\leq F(\lambda^k) - \frac{1}{L n N_{\max}} \normtwo{\nabla F(\lambda^k)}^2 
    \numberthis \label{eq:progress_SU-CD}
\end{align*}
where $N_i \; {=} \; \abs{\Scal_i}$, $N_{\max}  \; {=} \; \max_i N_i$, and the factor 2 in the denominator disappears because each coordinate $\ell \equiv (i,j)$ is counted twice in the sums of $\Scal_i$ and $\Scal_j$, respectively.

The standard procedure to show the linear convergence of CD in the single-machine case is to lower-bound $\normtwo{\nabla F(\lambda)}^2$ using the strong convexity of the function \cite{nesterov2012efficiency, nutini2015coordinate}.
However, since $F$ is \emph{not} strongly convex (Lemma \ref{lemma:no_SC}), we cannot apply this procedure to get the linear rate of SU-CD.

We can, however, use $F$'s strong convexity in $\normAA{\cdot}$ instead (Lemma \ref{lemma:SC_in_AA}).
The next result gives the core of the proof.

\begin{lemma} \label{lemma:equality_of_norms}
    It holds that 
    \begin{equation} \label{eq:equality_norms_2_and_A}
        \normtwo{\nabla F(\lambda)} = \normAA{\nabla F(\lambda)} = \normAA{\nabla F(\lambda)}^*,
    \end{equation}
    where $\normAA{\cdot}^*$ is the dual norm of $\normAA{\cdot}$, defined as (e.g. \cite{boyd2004convex})
    \begin{equation} \label{eq:dual_norm_AA}
        \normAA{z}^* = \sup_{x \in \Rbb^d} \Cbraces{ z^T x \biggr\rvert \normAA{x} \leq 1 }.
    \end{equation}
\end{lemma}

\begin{proof}
    Note that 
    $\forall w \neq \zeros$ such that $F(\lambda + tw) = F(\lambda) \; \forall t$,
    it holds that $w^T \nabla F(\lambda) = 0$ 
    and thus $\nabla F(\lambda) \in \range(A^T)$.
    This means that $\AA \nabla F(\lambda) = I_E \nabla F(\lambda)$, and therefore it holds that
    $\normAA{\nabla F(\lambda)} = \normtwo{\nabla F(\lambda)}$.
    Finally, since the dual norm of the L2 norm is the L2 norm itself, we have that also $\normAA{\nabla F(\lambda)}^* = \normtwo{\nabla F(\lambda)}$, which gives the result.
\end{proof}

We now use Lemma \ref{lemma:equality_of_norms} to prove the linear rate of SU-CD.
Note that since $\lambda^*$ is an optimum of $F$, convergence is exact. 
This is valid not only for SU-CD, but for all other algorithms presented in this paper.

%
\setlength{\textfloatsep}{10pt} 
\begin{algorithm}[t]
\caption{Setwise Uniform CD (SU-CD)}
\label{alg:SU-CD}
\begin{algorithmic}[1]
    \MYSTATE {\bfseries Input:} Functions $f_i$, step $\eta$, incidence matrix $A$, graph~$\Gcal$
    \MYSTATE Initialize $\theta_i^0, i=1,\ldots,n$ and $\lambda_\ell^0, \ell=1,\ldots,E$
    \MYSTATE \textbf{for} $k = 1,2,\ldots$ \textbf{do}
        \MYSTATE \algindent Sample activated node $i \in \{1,\ldots,n\}$ uniformly
        \MYSTATE \algindent Node $i$ picks neighbor $j \gets \Ucal \{h: h \in \Ncal_i \}$
        \MYSTATE \algindent Node $i$ computes $\nabla f_i^*(u_i^T A \lambda)$ and sends it to $j$
        \MYSTATE \algindent Node $j$ computes $\nabla f_j^*(u_j^T A \lambda)$ and sends it to $i$
        \MYSTATE \algindent \parbox[t]{\dimexpr\linewidth-\algorithmicindent}{Nodes $i,j{:}\;(i,j)\equiv\ell$ use \eqref{eq:coord_gradient} to update their local copies of $\lambda_\ell$ by  
        $\lambda_\ell^k \gets \lambda_\ell^{k-1} - \eta \nabla_\ell F(\lambda)$} \label{line:SU_update} \\
        \MYSTATE \algindent $\lambda_m^k \gets \lambda_m^{k-1} \; \forall \text{ edges } m \neq \ell$
\end{algorithmic}
\end{algorithm}

\begin{theorem}[\textit{\textbf{Rate of SU-CD}}] \label{theo:rate_SU-CD}
    SU-CD converges as
    \begin{equation*}
        \Exp{F(\lambda^{k+1}) \mid \lambda^k} - F(\lambda^*) \leq
            \paren{1 - \frac{2\sigma_A}{LnN_{\max}}} \Sbraces{F(\lambda^k) - F(\lambda^*)}. 
    \end{equation*}
\end{theorem}

\begin{proof}
    Since $F(\lambda)$ is strongly convex in $\normAA{\cdot}$ with strong convexity constant $\sigma_A$ (Lemma \ref{lemma:SC_in_AA}), it holds
    \[ F(y) \geq F(x) + \Abraces{\nabla  F(x), y - x} + \frac{\sigma_A}{2} \normAA{y-x}^2. \]
    
    Minimizing both sides with respect to $y$ as in \cite{nutini2015coordinate} we get 
             
    \begin{equation} \label{eq:guarantee_suboptim_SU}
        F(x^*) \geq F(x) - \frac{1}{2\sigma_A} \paren{ \normAA{\nabla  F(x)}^* }^2,
    \end{equation}

    \noindent
    and rearranging terms we obtain the lower bound 
    $\paren{\normAA{\nabla F(x)}^*}^2 \geq 2\sigma_A (F(x) - F(x^*))$.
    
    Finally, we can use Lemma \ref{lemma:equality_of_norms} to replace 
    $\norm{\nabla  F(x)}_2^2$ with $\paren{\normAA{\nabla F(x)}^*}^2$
    in \eqref{eq:progress_SU-CD}, and use the lower bound on $\paren{\normAA{\nabla F(x)}^*}^2$ to get the result.
\end{proof}

Note that vector $\lambda$ has $\frac{1}{2} \sum_i N_i = E \leq \frac{nN_{\max}}{2}$ coordinates, where the inequality holds with equality for regular graphs. 
We make the following remark.

\begin{remark}
    If $\Gcal$ is regular, the linear convergence rate of SU-CD is $\frac{\sigma_A}{LE}$, which matches the rate of single-machine uniform CD for strongly convex functions \cite{nesterov2012efficiency, nutini2015coordinate}, with the only difference that now the strong convexity constant $\sigma_A$ is defined over norm $\normAA{\cdot}$.
\end{remark}

In the next section we analyze SGS-CD and show that its convergence rate can be up to $N_{\max}$ times that of SU-CD.

%
%

\subsection{Setwise Gauss-Southwell CD (SGS-CD)}

%
\setlength{\textfloatsep}{10pt} 
\begin{algorithm}[t]
\caption{Setwise Gauss-Southwell CD (SGS-CD)}
\label{alg:SGS-CD}
\begin{algorithmic}[1]
    \MYSTATE {\bfseries Input:} Functions $f_i$, step $\eta$, incidence matrix $A$, graph~$\Gcal$
    \MYSTATE Initialize $\theta_i^0, i=1,\ldots,n$ and $\lambda_\ell^0, \ell=1,\ldots,E$
    
    \MYSTATE \textbf{for} $k = 1,2,\ldots$ \textbf{do}
        \MYSTATE \algindent Sample activated node $i \in \{1,\ldots,n\}$ uniformly

        \MYSTATE \algindent All $h \in \Ncal_i$ send their $\nabla f_h^*(u_h^T A \lambda)$ to node $i$\label{line:all_neighbors_send_grads}\\
        
        \MYSTATE\algindent\parbox[t]{\dimexpr\linewidth-\algorithmicindent}{
        Node $i$ computes $\nabla_\ell F(\lambda) \; \forall \ell \in \Scal_i$ (equivalently, $\forall h \in \Ncal_i)$ with 
        \eqref{eq:coord_gradient} using the received $\nabla f_h^*(u_h^T A \lambda)$}\\

        
        \MYSTATE \algindent Node $i$ selects $j \gets \max_{h \in \Ncal_i} \abs{\nabla_\ell F(\lambda)}, 
        \ell {\equiv} (i,h)$
        \MYSTATE \algindent Node $i$ sends $\nabla f_i^*(u_i^T A \lambda)$ to $j$
        \MYSTATE \algindent \parbox[t]{\dimexpr\linewidth-\algorithmicindent}{Nodes $i,j{:}\;(i,j)\equiv\ell$ use \eqref{eq:coord_gradient} to update their local copies of $\lambda_\ell$ by  
        $\lambda_\ell^k \gets \lambda_\ell^{k-1} - \eta \nabla_\ell F(\lambda)$} \label{line:SGS_update} \\
        \MYSTATE \algindent Nodes $i$ and $j$ recompute their $\nabla f_h^*(u_h^T A \lambda), h \in \{i,j\}$\label{line:recompute_grads}\\
        \MYSTATE \algindent $\lambda_m^k \gets \lambda_m^{k-1} \; \forall \text{ edges } m \neq \ell$
\end{algorithmic}
\end{algorithm}

In SGS-CD, as shown in Algorithm \ref{alg:SGS-CD}, the activated node $i$ selects the neighbor $j$ to contact applying the GS rule within the edges in~$\Scal_i$:
 
\begin{equation*}
    \ell = \argmax_{m \in \Scal_i} \paren{\nabla_m F(\lambda)}^2, 
\end{equation*} 
  
\noindent
and then $j$ is the neighbor that satisfies $\ell \equiv (i,j)$.
In order to make this choice, all nodes $h \in \Ncal_i$ must send their $\nabla f_h^*(u_h^T A \lambda)$ to node $i$ (line \ref{line:all_neighbors_send_grads} in Algorithm \ref{alg:SGS-CD}). 
Remark that the $\nabla f_h^*$ only need to be recomputed by the nodes that complete the update (line \ref{line:recompute_grads}), and not by all neighbors originally pinged by the activated node. 
This keeps the computation complexity of SGS-CD equal to that of SU-CD, and only communication has increased. 
We discuss this fact in Section \ref{sec:conclusion}.

To obtain the convergence rate of SGS-CD we will follow the steps taken for SU-CD in the proof of Theorem \ref{theo:rate_SU-CD}.
As done for SU-CD, we start by computing the per-iteration progress taking expectation in \eqref{eq:iter_progress} for SGS-CD:



\begin{equation} \label{eq:progress_SGS-CD}
    \Exp{F(\lambda^{k+1}) \mid \lambda^k} \leq 
    F(\lambda^k) - \frac{1}{2Ln} \sum_{i=1}^n \max_{\ell \in \Scal_i} \paren{\nabla_\ell F(\lambda^k)}^2. \hspace{-4pt}
\end{equation} 


Given this per-iteration progress, to proceed as we did for SU-CD we need to show that
(i) the sum on the right-hand side of \eqref{eq:progress_SGS-CD} defines a norm, 
and (ii) strong convexity holds in its dual norm. 
We start by defining the selector matrices $T_i$, which will significantly simplify notation.

\begin{definition}[\textit{\textbf{Selector matrices}}] \label{def:selector_matrices}
    The selector matrices $T_i \in \Cbraces{0,1}^{N_i \times E}, i=1,\ldots,n$ select the coordinates of a vector in $\Rbb^E$ that belong to set $\Scal_i$. 
    Note that any vertical stack of the unitary vectors $\Cbraces{e_\ell^T}_{\ell \in \Scal_i}$ gives a valid $T_i$.
\end{definition}

We can now show that the sum in \eqref{eq:progress_SGS-CD} is a (squared) norm. 
Since the operation involves applying $\max(\cdot)$ within each set $\Scal_i$, we will denote this norm $\normSM{x}$, where the subscript SM stands for ``Set-Max".

\begin{lemma} 
    The function $\normSM{x}~\triangleq~\sqrt{\sum_{i=1}^n \norm{T_i x}_\infty^2} = \sqrt{\sum_{i=1}^n \max_{j \in \Scal_i} x_j^2}$ is a norm in $\Rbb^E$.
\end{lemma}

\begin{proof}
    Using $\max_{j \in \Scal_i} \paren{x_j^2 {+} y_j^2} {\leq} \max_{j \in \Scal_i} x_j^2 {+} \max_{j \in \Scal_i} y_j^2$ 
    and $\sqrt{a+b} \leq \sqrt{a} + \sqrt{b}$ we can show that $\normSM{\cdot}$ satisfies the triangle inequality.
    It is straightforward to show that $\normSM{\alpha x} = \abs{\alpha} \normSM{x}$ and 
    $\normSM{x} = 0$ if and only if $x = \zeros$.
\end{proof}

Following the proof of Theorem \ref{theo:rate_SU-CD}, we would like to show that $F$ is strongly convex in the dual norm $\normSM{\cdot}^*$. 
Furthermore, we would like to compare the strong convexity constant $\sigma_\SM$ with $\sigma_A$ to quantify the speedup of SGS-CD with respect to SU-CD. 
It turns out, though, that computing $\normSM{\cdot}^*$ is not easy at all; the main difficulty stems from the fact that the sets $\Scal_i$ are overlapping (or non-disjoint), since each coordinate $\ell \equiv (i,j)$ belongs to both $\Scal_i$ and $\Scal_j$. The first scheme in Figure \ref{fig:example_sets} illustrates this fact for the 3-node clique. 

\begin{figure}[t]
    \centering
    \includegraphics[trim={13cm 6cm 32cm 7cm}, clip, width=0.6\linewidth]{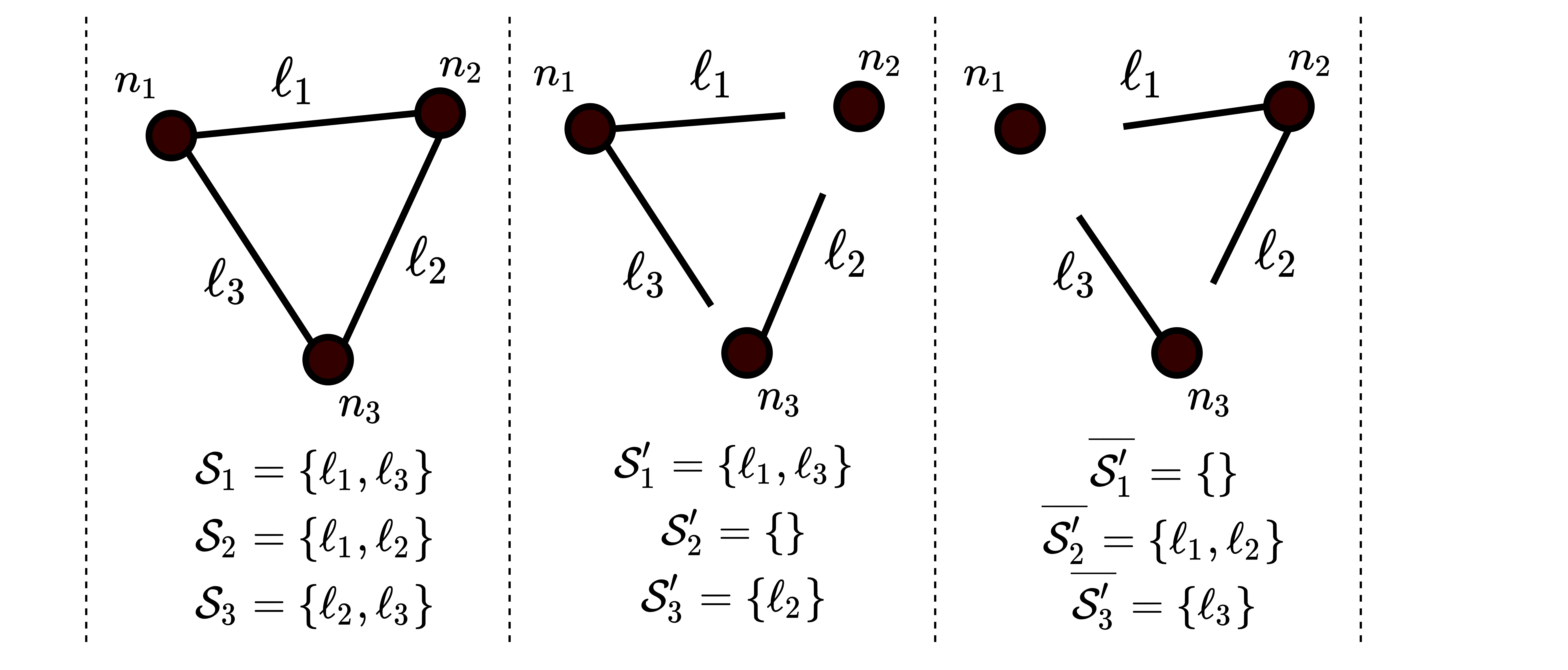}
    \caption{Example of sets $\Scal_i$ and \emph{one possibility} for $\Scal_i'$ and $\overline{\Scal_i'}$}
    \label{fig:example_sets}
\end{figure}

To circumvent this issue, we define a new norm  $\normSMNO{\cdot}^*$ (``Set-Max Non-Overlapping") that we can directly relate to  $\normSM{\cdot}^*$ (Lemma \ref{lemma:ineq_SM_SMNO}) and whose value we can compute explicitly (Lemma \ref{lemma:value_norm_SMNO}), which will later allow us to relate the three strong convexity constants  $\sigma_\SM, \sigma_\SMNO$, and $\sigma_A$ (Theorem \ref{theo:rate_SGS-CD}).

\begin{definition}[\textit{\textbf{Norm $\normSMNO{\cdot}^*$}}] \label{def:norm_SMNO}   
    We assume that each coordinate $\ell \equiv (i,j)$ is assigned to only one of the sets $\Scal_i' \subseteq \Scal_i$ or
    $\Scal_j' \subseteq \Scal_j$, such that the new sets $\Cbraces{\Scal_i'}_{i=1}^n$ are non-overlapping (some sets can be empty), and all coordinates $\ell$ belong to \textit{exactly one} set in $\Cbraces{\Scal_i'}$. 
    We name the selector matrices of these new sets $T_i'$, so that each possible choice of $\Cbraces{\Scal_i'}$ defines a different set $\{T_i'\}$.
    Then, we define
    \begin{equation} \label{eq:def_dual_norm_SMNO}
         \normSMNO{z}^* = \sup_{x} \Cbraces{ z^T x \; \biggr\rvert 
         \sqrt{ \sum_{i=1}^n \norm{T_i^* x}_\infty^2} \leq 1 },
    \end{equation}
    with the choice of non-overlapping sets
    \begin{equation} \label{eq:optimal_Ti}
         \{T_i^*\} = 
         \arg \max_{ \{T_i'\} } \; \sum_{i=1}^n \norm{T_i' x}_\infty^2.
    \end{equation}
    
    Note that the maximizations in \eqref{eq:def_dual_norm_SMNO} and \eqref{eq:optimal_Ti} are coupled. 
    We denote the value of $x$ that attains \eqref{eq:def_dual_norm_SMNO} by $x_{\SMNO}^*$.
\end{definition}
    
The definition of sets $\Scal_i'$ corresponds to assigning each edge $\ell$ to one of the two nodes at its endpoints, as illustrated in the second scheme of Figure~\ref{fig:example_sets}. 
Therefore, for each possible pair $\paren{\Cbraces{\Scal_h'}, \Cbraces{T_h'}}, h \in [n]$ we can define a complementary pair $(\{\overline{\Scal_h'}\}, \{\overline{T_h'}\})$ 
such that if $\ell \equiv (i,j)$ was assigned to $\Scal_i'$ in $\Cbraces{\Scal_h'}$, then it is assigned to $\overline{\Scal_j'}$ in $\{\overline{\Scal_h'}\}$.
This corresponds to assigning $\ell$ to the opposite endpoint (node) to the one originally chosen, as shown in the third scheme of Figure~\ref{fig:example_sets}.
With these definitions, it holds (potentially with some permutation of the rows):
\[ T_i = \begin{bmatrix} T_i' \\ \overline{T_i'} \end{bmatrix} =
    \begin{bmatrix} T_i' \\ \zeros \end{bmatrix} + 
    \begin{bmatrix} \zeros \\ \overline{T_i'} \end{bmatrix}, \; i=1,\ldots,n. \]

We remark that the equality above holds for \emph{any} $\{T_i'\}$ corresponding to a feasible  assignment $\Cbraces{\Scal_i'}$, and in particular it hols for $(\Cbraces{\Scal_i^*}, \{T_i^*\})$.   
This fact is used in the proof of the following lemma, which relates norms $\normSM{\cdot}^*$ and $\normSMNO{\cdot}^*$. 
This will allow us to complete the analysis with $\normSMNO{\cdot}^*$, which we can compute explicitly (Lemma \ref{lemma:value_norm_SMNO}).

\begin{lemma} \label{lemma:ineq_SM_SMNO}
    The dual norm of $\normSM{\cdot}$, denoted $\normSM{\cdot}^*$, satisfies
    $\frac{1}{2} \paren{\normSMNO{z}^*}^2 \leq \paren{\normSM{z}^*}^2 \leq \paren{\normSMNO{z}^*}^2$.
\end{lemma}

\begin{proof}
    By definition 
    \begin{equation} \label{eq:def_norm_dual_SM}
    \normSM{z}^* = \sup_x \Cbraces{ z^T x \; \biggr\rvert \sqrt{ \sum_{i=1}^n \norm{T_i x}_\infty^2 } \leq 1 }. 
    \end{equation}

    By inspection we can tell that the $x$ that attains the supremum, denoted $x_{\SM}^*$, will satisfy 
    $\sum_{i=1}^n \norm{T_i x_{\SM}^*}_\infty^2 = 1$.
    Similarly, $x_{\SMNO}^*$ (defined under \eqref{eq:optimal_Ti}) must satisfy
    $\sum_{i=1}^n \norm{T_i^* x_{\SMNO}^*}_\infty^2 = 1$. 
    Note that in these two equalities the $\{T_i\}$ are overlapping sets and the $\{T_i^*\}$ are non-overlapping. 
    Therefore, in order to satisfy both equalities it must hold that 
    $|[x_{\SM}^*]_\ell| \leq |[x_{\SMNO}^*]_\ell| , \ell \in [E]$, i.e. the magnitude of the entries of $x_{\SMNO}^*$ are equal or larger than the magnitudes of the corresponding entries of $x_{\SM}^*$.
    Referring to the definitions \eqref{eq:def_dual_norm_SMNO} and \eqref{eq:def_norm_dual_SM}, this means that $\normSM{z}^* \leq \normSMNO{z}^*$.
    
    We now proceed to show the first inequality in the lemma. We note that      
    \begin{equation} \label{eq:ineq_Ti}
        \sum_{i=1}^n \norm{T_i x}_\infty^2 
        = \sum_{i=1}^n \norm{\begin{bmatrix} T_i' \\ \zeros \end{bmatrix} x + 
        \begin{bmatrix} \zeros \\ \overline{T_i'} \end{bmatrix} x}_\infty^2
        \leq \sum_{i=1}^n \norm{T_i' x}_\infty^2 + \sum_{i=1}^n \norm{\overline{T_i'} x}_\infty^2
        \leq 2 \sum_{i=1}^n \norm{\widehat{T_i'} x}_\infty^2,
    \end{equation}    
    with 
    \begin{equation} \label{eq:max_over_Ti}
        \{\widehat{T_i'}\} = \arg\max_{\{T_i'\},\{\overline{T_i'}\}} 
    \paren{\sum_{i=1}^n \norm{T_i' x}_\infty^2, \sum_{i=1}^n \norm{\overline{T_i'} x}_\infty^2}.
    \end{equation}

    We now evaluate \eqref{eq:ineq_Ti} and \eqref{eq:max_over_Ti} at $x_{\SMNO}^*$. 
    Due to \eqref{eq:optimal_Ti} we have $\{\widehat{T_i'}\} = \{T_i^*\}$, and since  $\sum_{i=1}^n \norm{T_i^* x_{\SMNO}^*}_\infty^2 = 1$, the rightmost member of \eqref{eq:ineq_Ti} takes value 2.
    Then, dividing both sides of \eqref{eq:ineq_Ti} by 2 we obtain     
    \[ \frac{1}{2} \sum_{i=1}^n \norm{T_i x_{\SMNO}^*}_\infty^2 = 
    \sum_{i=1}^n \norm{T_i \frac{x_{\SMNO}^*}{\sqrt{2}}}_\infty^2 \leq 1, \]
    and since $\sum_{i=1}^n \norm{T_i x_{\SM}^*}_\infty^2 = 1$, we conclude that it must hold that 
    $\frac{1}{\sqrt{2}} |[x_{\SMNO}^*]_\ell| \leq |[x_{\SM}^*]_\ell|, \ell \in [E]$, 
    and thus $\frac{1}{\sqrt{2}} \normSMNO{z}^* \leq \normSM{z}^*$.
\end{proof}

The next lemma gives the value of $\normSMNO{x}^*$ explicitly, which will be needed to compare the strong convexity constant $\sigma_{\SMNO}$ with $\sigma_A$.

\begin{lemma} \label{lemma:value_norm_SMNO}
    It holds that $\normSMNO{x}^* = \sqrt{\sum_{i=1}^n \norm{T_i^* x}_1^2}$.
\end{lemma}

\begin{proof}
    Since the sets $\{\Scal_i^*\}$ are non-overlapping and in \eqref{eq:def_dual_norm_SMNO} norm $\norm{\cdot}_\infty$ is applied per-set, the entries $x_\ell$ of $x_{\SMNO}^*$ will have $\abs{x_\ell} = x^{(i)} \geq 0 \; \forall \; \ell \in \Scal_i^*$ and the sign will match that of the entries of $z$, i.e. $\sign(x_\ell) = \sign(z_\ell)$. 
    The maximization of \eqref{eq:def_dual_norm_SMNO} then becomes 
     
    \begin{equation*}
        \begin{aligned}
        & \maximize_{\Cbraces{x^{(i)}}} && \sum_{i=1}^n \sum_{\ell \in \Scal_i^*}  \paren{ \abs{z_\ell} \cdot x^{(i)} } \\
        & \st && \sqrt{\sum_{i=1}^n \paren{x^{(i)}}^2} \leq 1.
        \end{aligned}
    \end{equation*}
    
    Factoring out $x^{(i)}$ in the objective and noting that 
    $\sum_{\ell \in \Scal_i^*} \abs{z_\ell} = \norm{T_i^* z}_1$, we can define now both 
    $w = [x^{(1)}, \ldots, x^{(n)}]^T$ and $y = \Sbraces{ \norm{T_1^* z}_1, \ldots, \norm{T_n^* z}_1 }^T$ 
    so that \eqref{eq:def_dual_norm_SMNO} reads
     
    \[ \normSMNO{z}^* = \sup_w \Cbraces{ y^T w \; \biggr\rvert \normtwo{w} \leq 1}. \]
    
    The right-hand side is the definition of $\normtwo{\cdot}^*$, the dual of the L2 norm, evaluated at $y$. Since $\normtwo{\cdot}^* = \normtwo{\cdot}$, we have that
    $ \normSMNO{z}^* = \normtwo{y} = \sqrt{\sum_{i=1}^n \norm{T_i^* z}_1^2} $.        
\end{proof}

We can now prove the linear convergence rate of SGS-CD.

\begin{theorem}[\textit{\textbf{Rate of SGS-CD}}] \label{theo:rate_SGS-CD}
    SGS-CD converges as
    \begin{equation*} 
         \Exp{F(\lambda^{k+1}) \mid \lambda^k} - F(\lambda^*) \leq
        \paren{1 - \frac{\sigma_{\SM}}{Ln}} \Sbraces{F(\lambda^k) - F(\lambda^*)}, 
    \end{equation*}
    with 
    \begin{equation} \label{eq:ineq_sigmaSM_sigmaA}
      \frac{\sigma_A}{N_{\max}} \leq \sigma_{\SM} \leq 2 \sigma_A.  
    \end{equation}
\end{theorem}

\begin{proof}
    Similarly to what we did for SU-CD, we can depart from the strong convexity of $F$ in the $\normSM{\cdot}$ norm:
    \[ F(y) \geq F(x) + \Abraces{\nabla F(x), y-x} + \frac{\sigma_{\SM}}{2} \paren{\normSM{y-x}^*}^2, \]
    then minimize both sides with respect to $y$ to obtain 
    \begin{equation} \label{eq:guarantee_suboptim_SM}
        F(x^*) \geq F(x) - \frac{1}{2\sigma_{\SM}} \paren{\normSM{\nabla F(x)}}^2, 
    \end{equation}
    which is analogous to \eqref{eq:guarantee_suboptim_SU}, and then rearrange terms to obtain a lower bound on $\normSM{\nabla F(\lambda)}^2$. 
    Using this lower bound in \eqref{eq:progress_SGS-CD} gives the rate of SGS-CD.

    Since this rate is given in terms of $\sigma_{\SM}$ and that of SU-CD in Theorem \ref{theo:rate_SU-CD} is given in terms of $\sigma_A$, we need \eqref{eq:ineq_sigmaSM_sigmaA} to compare both rates.
    However, we cannot prove these inequalities directly because we cannot compare norms $\normAA{\cdot}$ and $\normSM{\cdot}^*$ (due to the overlap of the coordinate sets, which prevents us from computing the latter). 
    However, we \textit{can} compare $\normAA{\cdot}$ with $\normSMNO{\cdot}^*$ and $\normSM{\cdot}^*$ with $\normSMNO{\cdot}^*$ individually, from which we will obtain \eqref{eq:ineq_sigmaSM_sigmaA}.
    In particular, we will show the inequalities 
    \begin{equation} \label{eq:ineq_sigmaSMNO_sigmaA}
      \frac{\sigma_A}{N_{\max}} \leq \sigma_{\SMNO} \leq \sigma_A
    \end{equation}
    and 
    \begin{equation} \label{eq:ineq_sigmaSMNO_sigmaSM}
        \sigma_{\SMNO} \leq \sigma_{\SM} \leq 2 \sigma_{\SMNO}.  
    \end{equation}

    We start by proving \eqref{eq:ineq_sigmaSMNO_sigmaA}.
    Below we assume $x \in \range(A^T)$; the results here can then be directly applied to the proofs above because $\normAA{\cdot}, \normSM{\cdot}, \normSMNO{\cdot}$ and their duals are applied to $\nabla F$, which is always in $\range(A^T)$ (Lemma \ref{lemma:equality_of_norms}). 

    For $x \in \range(A^T)$ it holds that 
    (Lemmas \ref{lemma:equality_of_norms} and \ref{lemma:value_norm_SMNO}):
    \begin{gather*}
        \normAA{x}^2 = \normtwo{x}^2 = \sum_{i=1}^E x_i^2 = \sum_{i=1}^n \normtwo{T_i^* x}^2 \\
         \paren{\normSMNO{x}^*}^2 = \sum_{i=1}^n \norm{T_i^* x}_1^2.
    \end{gather*}
      
    We also note that, using the Cauchy-Schwarz inequality and denoting $[v]_i$ the $i$\ts{th} entry of vector $v$, it holds both that
    \begin{gather*}
        \sum_{i=1}^n \normtwo{T_i^* x}^2 \leq 
        \sum_{i=1}^n \Bigg( \sum_{j \in \Scal_i^*} |x_j| \Bigg)^2 =
        \sum_{i=1}^n \norm{T_i^* x}_1^2, 
        \text{ and} \\
        \sum_{i=1}^n \norm{T_i^* x}_1^2 
        = \sum_{i=1}^n \paren{ \ones^T 
        \bigg[ \Big| [T_i^*x]_1 \Big| , \ldots, \Big| [T_i^*x]_{N_i^*} \Big| \bigg]^T }^2 \\
        \stackrel{\text{C.S.}}{\leq}
        \sum_{i=1}^n N_i^* \normtwo{T_i^*x}^2
        \leq N_{\max} \sum_{i=1}^n \normtwo{T_i^*x}^2,
    \end{gather*}
    where $N_i^* = \abs{\Scal_i^*}$.
    We can summarize these relations as 
    \[ \frac{1}{N_{\max}} \paren{\normSMNO{x}^*}^2
    \leq \normAA{x}^2 \leq  \paren{\normSMNO{x}^*}^2. \]
    
    Using these inequalities in the strong convexity definitions, similarly to \cite{nutini2015coordinate}, we get both
    \begin{equation} 
    \begin{aligned} \label{eq:lower_ineq} 
        &F(y) \geq F(x) + \Abraces{\nabla F(x), y-x} + \frac{\sigma_A}{2} \paren{\normAA{y-x}}^2 \\
        &\geq F(x) + \Abraces{\nabla F(x), y-x} + \frac{\sigma_A}{2N_{\max}} \paren{\normSMNO{y-x}^*}^2, 
    \end{aligned}
    \end{equation} 
    and
    \begin{equation} 
    \begin{aligned} \label{eq:higher_ineq}
        F&(y) \geq F(x) + \Abraces{\nabla F(x), y - x} + \frac{\sigma_{\SMNO}}{2} \paren{\normSMNO{y-x}^*}^2 \\
        &\geq F(x) + \Abraces{\nabla F(x), y-x} + \frac{\sigma_{\SMNO}}{2} \paren{\normAA{y - x}}^2.
    \end{aligned}
    \end{equation} 
    
    Equation \eqref{eq:lower_ineq} says that $F$ is at least $\frac{\sigma_A}{N_{\max}}$-strongly convex in $\normSMNO{\cdot}^*$, and eq. \eqref{eq:higher_ineq} says that $F$ is at least $\sigma_{\SMNO}$-strongly convex in $\normAA{\cdot}$.
    Together they imply \eqref{eq:ineq_sigmaSMNO_sigmaA}. 
    
    We can show \eqref{eq:ineq_sigmaSMNO_sigmaSM} by the same procedure applied in eqs. \eqref{eq:lower_ineq} and \eqref{eq:higher_ineq}, but now using the strong convexity of $F$ in norms $\normSM{\cdot}^*$ and $\normSMNO{\cdot}^*$ together with Lemma \ref{lemma:ineq_SM_SMNO}. 
    From $\frac{1}{2} (\normSMNO{z}^*)^2 \leq (\normSM{z}^*)^2$ we get $\frac{1}{2} \sigma_{\SM} \leq \sigma_{\SMNO}$,
    and from $(\normSM{z}^*)^2 \leq (\normSMNO{z}^*)^2$ we get $\sigma_{\SMNO} \leq \sigma_{\SM}$.

    Finally, putting \eqref{eq:ineq_sigmaSMNO_sigmaA} and \eqref{eq:ineq_sigmaSMNO_sigmaSM} together gives \eqref{eq:ineq_sigmaSM_sigmaA}.
\end{proof}

Theorems \ref{theo:rate_SU-CD} and \ref{theo:rate_SGS-CD} together allow us to compare the convergence rates of SU-CD and SGS-CD. 
We note that when $\sigma_{\SM}$ takes the upper value in \eqref{eq:ineq_sigmaSM_sigmaA}, SGS-CD is (in expectation) $N_{\max}$ times faster than SU-CD. 
The lower bound in \eqref{eq:ineq_sigmaSM_sigmaA}, on the other hand, suggests that SGS-CD could be slower than SU-CD.
We remark that (in expectation) this is not true and the lower bound is vacuous, since the following holds. 

\begin{remark} \label{rem:SGS_faster_than_SU}
For the same sequence of node activations, the suboptimality reduction of SGS-CD at each iteration is equal to or larger than that of SU-CD.
\end{remark}

Taking this fact into account, we have the following result. 

\begin{corollary} \label{cor:SGS_faster_SU}
    In expectation, SGS-CD converges at least as fast as SU-CD, and can be up to $N_{\max}$ times faster.
\end{corollary}

Note that this result is analogous to that of \cite{nutini2015coordinate} for single-machine CD, where they show that the GS rule can be up to $d$ times faster than uniform sampling, $d$ being the dimensionality of the problem.

We remark that achieving the upper bound of $N_{\max}$ speedup may require designing a scenario particularly favorable to SGS-CD with respect to SU-CD. 
Similarly, having both algorithms converge at the same speed requires a particularly adversarial setting, but this is not hard to design: if all $f_i$ are identical and the graph is complete, no particular benefits can be expected from SGS-CD over SU-CD.

In our simulations of Section \ref{sec:numerical_results} for the decentralized setting, SGS-CD achieves a speedup approximately in the middle of the range between 1 and $N_{\max}$. 
We show that this speedup \textit{increases linearly with $N_{\max}$}, achieving remarkable gains in terms of suboptimality reduction versus number of iterations (see Figure \ref{fig:decen_and_distrib}). 
Furthermore, in the same figure we show that for the parallel distributed setting \textit{the maximum speedup of $N_{\max}$ is attainable}. 
We explain this further in Section \ref{sec:parallel_distributed}.

%
%

\section{Setwise Lipschitz CD Algorithms} \label{sec:algos_with_Lips}

In Section \ref{sec:algorithms} we stated that the dual function $F$ is $L$-smooth and therefore a sufficient condition for the setwise algorithms to converge was using stepsize $\eta = 1/L$. 
However, the updates of some (and maybe \textit{many}) coordinates could use larger stepsizes by exploiting the fact $F$ has coordinate-wise smoothness $L_\ell \leq L$, i.e. for $\alpha \in \Rbb$:
\begin{equation} \label{eq:coord_Lips_constants}
    \abs{\nabla_\ell F(\lambda + \alpha e_\ell) - \nabla_\ell F(\lambda)} \leq L_\ell \alpha.
\end{equation}

Therefore, when the  coordinate-wise Lipschitz constants $L_\ell$ are known or can be estimated (see Section \ref{ref:L_estimation}) we can apply the update \eqref{eq:CD_update} with stepsize $\eta^k = 1/L_\ell$, with $\ell$ being the coordinate updated at iteration $k$. 

In the sections that follow we show that by using the knowledge (or estimation) of the coordinate-wise Lipschitz constants and per-coordinate stepsizes we can have:
\begin{enumerate}
    \item An algorithm that has randomized but non-uniform neighbor selection that is provably faster than SU-CD. We call this algorithm Setwise Lipschitz CD (SL-CD).
    \item An algorithm that applies locally the Gauss-Southwell Lipschitz rule \cite{nutini2015coordinate} and that converges provably faster than both SL-CD and SGS-CD. We call this algorithm Setwise GS Lipschitz CD (SGSL-CD).
\end{enumerate}

Once again, while the seminal work of \cite{nutini2015coordinate} has analyzed both of these rules in the context of single-machine coordinate descent, their adaptation to setwise CD brings important new challenges. 
In this section, we prove that SL-CD is at least as fast as SU-CD, and that SGSL-CD is at least as fast as the fastest algorithm between SGS-CD and SL-CD. 

In the proofs that follow we will use the following fact. 
\begin{fact} \label{fact:dual_norm_of_per-coord_prod}
    Denote $a \circ b$ the per-entry product of vectors $a$ and $b$. Then, for any norm $\norm{\cdot}$ and finite $a: a_i > 0 \; \forall i$, if we define $\norm{x}_a := \norm{a \circ x}$, then $\norm{x}_a^* := \norm{a_{-1} \circ x}^*$ with $a_{-1} = [\frac{1}{a_1},\ldots,\frac{1}{a_d}]$.
\end{fact}

\begin{proof}
    By definition
    \begin{ceqn} \[ \norm{z}_a^* = \sup_{\norm{x}_a\leq1} z^T x, \] \end{ceqn}
    and defining $y := a \circ x$ we get
    \begin{ceqn}
    \[ \norm{z}_a^* = \sup_{\norm{y}\leq1} z^T (a_{-1} \circ y) 
    = \norm{a_{-1} \circ z}^*. \]
    \end{ceqn}
\end{proof}

%
%

\subsection{Setwise Lipschitz CD (SL-CD)}

In SL-CD, an activated node $i$ chooses the edge $\ell \in \Scal_i$ to update at random with probability 
\begin{equation} \label{eq:coord_samling_probs}
    p_\ell = \frac{L_\ell}{\sum_{m \in \Scal_i} L_m}
\end{equation}
and updates $\lambda_\ell$ applying \eqref{eq:CD_update} with stepsize $\eta^k = 1/L_\ell$.

For convenience, we define the quantities 
\[ L^{(i)} := \sum_{m \in \Scal_i} L_m \]
and
\[ \Lcal_\ell := \paren{\frac{1}{L^{(i)}} + \frac{1}{L^{(j)}}} 
\text{\quad for \quad}
\ell \equiv (i,j). \] 

With these definitions, and taking expectation in \eqref{eq:iter_progress} for the Lipschitz-dependent sampling probabilities \eqref{eq:coord_samling_probs} gives
\begin{align*} 
    \Ebb [F(\lambda^{k+1})  \mid \lambda^k ]
    &\leq F(\lambda^k) - \frac{1}{2} \Exp{\frac{1}{L_{\ell_k}} \Sbraces{\nabla_{\ell_k} F(\lambda^k)}^2} \\
    &= F(\lambda^k) - \frac{1}{2n} \sum_{i=1}^n \frac{1}{L^{(i)}} \sum_{\ell \in \Scal_i} \Sbraces{\nabla_\ell F(\lambda^k)}^2 \\
    &\stackrel{(a)}{=} F(\lambda^k) - \frac{1}{2n} \sum_{\ell=1}^E \Lcal_\ell \Sbraces{ \nabla_\ell F(\lambda^k)}^2 
\end{align*}
where in $(a)$ we used that $\ell \equiv (i,j)$ implies $\ell \in \Scal_i, \Scal_j$.

In order to prove the convergence rate of SL-CD, provided in Theorem \ref{theo:rate_SL-CD}, we define the norm 
\[ \norm{x}_\Lcal := \sqrt{\sum_{\ell=1}^E \Lcal_\ell x_\ell^2}, \] 
so that we can write the per-iteration progress of SL-CD as 
\begin{equation} \label{eq:per-iter_progress_SL}
    \Ebb [F(\lambda^{k+1})]
    \leq F(\lambda^k) - \frac{1}{2n} \norm{\nabla F(\lambda^k)}_\Lcal^2.
\end{equation}

Noting that $\norm{x}_\Lcal = \norm{x \circ \Sbraces{\sqrt{\Lcal_1},\ldots,\sqrt{\Lcal_E}} }_2$ we can apply Fact \ref{fact:dual_norm_of_per-coord_prod} to get its dual norm: 
\[ \norm{x}_\Lcal^* = \sqrt{\sum_{\ell=1}^E \frac{1}{\Lcal_\ell} x_\ell^2}. \]




We call $\sigma_\Lcal$ the strong convexity constant of $F$ in this norm:
\begin{equation} \label{eq:SC_in_dual_norm_Lcal}
    F(y) \geq F(x) + \Abraces{\nabla F(x), y-x} + \frac{\sigma_\Lcal}{2} (\norm{y-x}_{\Lcal}^*)^2.
\end{equation}

We use the definitions of $\norm{\cdot}_\Lcal^*$ and $\sigma_\Lcal$ in the proof of the linear rate of SL-CD, given in the theorem below.

\begin{theorem}[\textit{\textbf{Rate of SL-CD}}] \label{theo:rate_SL-CD}
    SL-CD converges as
    \begin{equation*} 
    \Exp{F(\lambda^{k+1}) \mid \lambda^k} - F(\lambda^*) \leq 
    \paren{1 - \frac{\sigma_\Lcal}{n}} \Sbraces{F(\lambda^k) - F(\lambda^*)}
    \end{equation*}
    and it holds that  
    \begin{ceqn} \begin{equation} \label{eq:inequalities_mu_L}
        \sigma_A \Lcal_{\min} \leq \sigma_\Lcal \leq \sigma_A \Lcal_{\max}
    \end{equation} \end{ceqn}
    with $\Lcal_{\min} = \min_\ell \Lcal_\ell$ and $\Lcal_{\max} = \max_\ell \Lcal_\ell$.
\end{theorem}
\begin{proof}
    We start by proving the linear rate.     
    Minimizing both sides of \eqref{eq:SC_in_dual_norm_Lcal} with respect to $y$ as done in \eqref{eq:guarantee_suboptim_SU} and \eqref{eq:guarantee_suboptim_SM} we get
    \begin{ceqn}
    \[ F(x^*) \geq F(x) - \frac{1}{2\sigma_\Lcal} \norm{\nabla F(x)}_{\Lcal}^2. \]
    \end{ceqn}
    Rearranging terms gives a lower bound on $\norm{\nabla F(x)}_{\Lcal}^2$, and replacing in \eqref{eq:per-iter_progress_SL} gives the result.    

    We now move on to show \eqref{eq:inequalities_mu_L}. 
    Once again, since the norms are evaluated at $\nabla F(\lambda)$ and \eqref{eq:equality_norms_2_and_A} holds, to obtain the relation between $\sigma_\Lcal$ and $\sigma_A$ we will compare $\norm{\cdot}_{\Lcal}^*$ with $\norm{\cdot}_2$ directly. 
    We have that 
    \[ c\norm{x}_2^2 - (\norm{x}_{\Lcal}^*)^2 = c\sum_\ell x_\ell^2 - \sum_\ell \frac{1}{\Lcal_\ell} x_\ell^2 
    = \sum_\ell \paren{c-\frac{1}{\Lcal_\ell}} x_\ell^2.  \]
    For $c \geq \max_\ell \frac{1}{\Lcal_\ell} = \frac{1}{\Lcal_{\min}}$ the expression is larger than zero, and thus
    \begin{equation} \label{eq:UB_norm_L}
        \frac{1}{\Lcal_{\min}} \norm{x}_2^2 \geq (\norm{x}_{\Lcal}^*)^2.
    \end{equation}
    Similarly, we have that  
    \[ c\norm{x}_{\Lcal}^2 - \norm{x}_2^2 = \sum_\ell \paren{\frac{c}{\Lcal_\ell}-1} x_\ell^2 \]
    is larger than zero for $c \geq \Lcal_{\max}$, and therefore
    \begin{equation} \label{eq:LB_norm_L}
        \Lcal_{\max} (\norm{x}_{\Lcal}^*)^2 \geq \norm{x}_2^2.
    \end{equation}

    Using these inequalities (and Lemma \ref{lemma:equality_of_norms}) in the strong convexity definitions we have on the one hand:
    \begin{align*} 
        f(y) &\geq f(x) + \Abraces{\nabla f(x), y-x} + \frac{\sigma_A}{2} \norm{y-x}_A^2 \\
        &\geq f(x) + \Abraces{\nabla f(x), y-x} + \frac{\sigma_A \Lcal_{\min}}{2} (\norm{y-x}_{\Lcal}^*)^2,
        \numberthis \label{eq:min_SC_sigma_A}
    \end{align*}
    and on the other hand:
    \begin{align*} 
        f(y) &\geq f(x) + \Abraces{\nabla f(x), y-x} + \frac{\sigma_\Lcal}{2} (\norm{y-x}_{\Lcal}^*)^2 \\
        &\geq f(x) + \Abraces{\nabla f(x), y-x} + \frac{\sigma_\Lcal}{2\Lcal_{\max}} \norm{y-x}_A^2.
        \numberthis \label{eq:min_SC_mu_L}
    \end{align*}
    Eqs. \eqref{eq:min_SC_sigma_A} and \eqref{eq:min_SC_mu_L} indicate respectively that $\sigma_\Lcal~\geq~\sigma_A \Lcal_{\min}$ and that $\sigma_A \geq \frac{\sigma_\Lcal}{\Lcal_{\max}}$. 
    Putting both together gives \eqref{eq:inequalities_mu_L}.    
\end{proof}

Having obtained the rate of SL-CD, we can compare it against that of SU-CD. 
We have the following result. 

\begin{corollary} \label{corol:SL_faster_SU}
    In expectation, SL-CD converges as fast or faster than SU-CD.
\end{corollary}
\begin{proof}
    The convergence rate of SU-CD is $\frac{2\sigma_A}{LnN_{\max}}$ (Theorem \ref{theo:rate_SU-CD}) and that of SL-CD is $\frac{\sigma_\Lcal}{n}$ (Theorem \ref{theo:rate_SL-CD}). 
    Since in the slowest case of SL-CD we have $\sigma_\Lcal = \sigma_A \Lcal_{\min}$, it suffices to show that $\Lcal_{\min} \geq \frac{2}{LN_{\max}}$.
    Indeed, we have that 
    \[  L^{(i)} = \sum_{\ell \in \Scal_i} L_\ell \leq L_{\max}  |\Scal_i| \leq L_{\max} N_{\max}  \]
    and therefore
    \[ \Lcal_{\min} = \min_{(i,j) \in \Ecal} \paren{\frac{1}{L^{(i)}} + \frac{1}{L^{(j)}}}  
    \geq \frac{2}{L_{\max} N_{\max}} \]
    The proof is complete by noting that it always holds that $L_{\max} \leq L$ \cite{wright2015coordinate}.
\end{proof}

Since both SL-CD and SGS-CD can converge at the same speed as SU-CD in the worst case, we cannot claim that either of them is faster than the other. 
We can, however, exploit the knowledge of the Lipschitz constants to get an improved version of the GS rule, known as the Gauss-Southwell Lipschitz rule \cite{nutini2015coordinate}, that when combined with per-coordinate stepsizes allows for faster convergence than both SGS-CD \textit{and} SL-CD.
We call this algorithm SGSL-CD, and we analyze it next.

%
%

\subsection{Setwise Gauss-Southwell Lipschitz CD (SGSL-CD)}

If node $i$ goes active, the Gauss-Southwell Lipschitz (GSL) rule chooses to update $\lambda_\ell, \ell \in \Scal_i$ according to
\[ \ell = \argmax_{m \in \Scal_i} \frac{\abs{\nabla_m f(x^k)}}{\sqrt{L_m}}. \]

If we now use the GSL rule with the per-coordinate stepsizes $\eta^k = 1/L_\ell$, the per-iteration progress given by \eqref{eq:iter_progress} becomes:
\begin{equation} \label{eq:SGSL-CD_progress}
\Exp{F(\lambda^{k+1})} \leq F(\lambda^k) - \frac{1}{2n} \sum_{i=1}^n \max_{\ell \in \Scal_i} \paren{\frac{1}{L_\ell} \Sbraces{\nabla_\ell F(\lambda^k)}^2}.
\end{equation}

Note the resemblance of this expression with the per-iteration progress of SGS-CD in \eqref{eq:progress_SGS-CD}. 
Similarly to the previous procedures, we define the ``Set-Max Lipschitz" norm:
\begin{equation} \label{eq:SML_norm}
\normSML{x} := \sqrt{ \sum_{i=1}^n \max_{\ell \in \Scal_i} \paren{\frac{1}{L_\ell} x_\ell^2} },
\end{equation}
and call $\sigma_{\SML}$ the strong convexity constant of $F$ in the dual norm $\normSML{\cdot}^*$
\begin{equation} \label{eq:SC_in_SML}
F(y) \geq F(x) + \Abraces{\nabla F(x), y-x} + \frac{\sigma_{\SML}}{2} \paren{\norm{y-x}_{\SML}^*}^2. 
\end{equation}
We can now state the convergence rate of SGSL-CD.

\begin{theorem}[\textit{\textbf{Rate of SGSL-CD}}] \label{theo:rate_SGSL-CD}
    SGSL-CD converges as
    \begin{equation} \label{eq:rate_SGSL-CD}
        \Exp{ F(\lambda^{k+1}) \mid \lambda^k } - F(\lambda^*) \leq
        \paren{1 - \frac{\sigma_{\SML}}{n}} \Sbraces{F(\lambda^k) - F(\lambda^*)},
    \end{equation}
    and it holds that 
    \begin{equation} \label{eq:SGSL_faster_than_SGS}
        \frac{\sigma_{\SM}}{L} \leq \sigma_{\SML}.
    \end{equation}
\end{theorem}

\begin{proof}
    The expression of the rate is obtained with the procedure followed for the previous algorithms: minimizing both sides of \eqref{eq:SC_in_SML} with respect to $y$ and arranging terms we can find 
    $ \norm{\nabla F(x)}_{\SML} \geq 2 \sigma_{\SML} (F(x) - F(x^*))$, and using this in \eqref{eq:SGSL-CD_progress} gives \eqref{eq:rate_SGSL-CD}. 

    We now show \eqref{eq:SGSL_faster_than_SGS}. 
    By definition, the dual norm of $\normSML{\cdot}$~is
    \begin{equation*}
        \normSML{z}^* :=  
        \sup_x \Cbraces{ z^T x \; \biggr\rvert \normSML{x} \leq 1} =
        \sup_x \Cbraces{ z^T x \; \biggr\rvert \sqrt{ \sum_{i=1}^n \max_{\ell \in \Scal_i} \paren{\frac{1}{L_\ell} x_\ell^2} } \leq 1 }.
    \end{equation*}
    
    Similarly, we can write the dual norm of $\normSM{\cdot}$ provided in \eqref{eq:def_norm_dual_SM} also as
    \begin{equation*}
        \normSM{z}^* :=  \sup_x \Cbraces{ z^T x \; \biggr\rvert \normSM{x} \leq 1} =
        \sup_x \Cbraces{ z^T x \; \biggr\rvert \sqrt{ \sum_{i=1}^n \max_{\ell \in \Scal_i} x_\ell^2 } \leq 1 }.
    \end{equation*}

    We call the values that achieve the supremum $x^*_{\SML}$ and $x^*_{\SM}$, respectively. 
    To maximize $z^T x$, these values will satisfy the constraints of each dual norm with equality, i.e. 
    \[ \normSML{x^*_{\SML}} = 1  \text{\quad and \quad} \normSM{x^*_{\SM}} = 1. \]

    From these conditions and the definitions of the dual norms above we obtain 
    \[ x^*_{\SML} \circ \Sbraces{ \frac{1}{\sqrt{L_1}}, \cdots, \frac{1}{\sqrt{L_E}} } = x^*_{\SM}. \]
    Furthermore, using again $L_{\max} \leq L$, we have
    \[ x^*_{\SM} = x^*_{\SML} \circ \Sbraces{ \frac{1}{\sqrt{L_1}}, \cdots, \frac{1}{\sqrt{L_E}} }
    \succeq \frac{1}{\sqrt{L_{\max}}} x^*_{\SML} \succeq \frac{1}{\sqrt{L}} x^*_{\SML}, \]
    where ``$\succeq$" indicates coordinate-wise inequality,  and therefore  
    \[ \normSM{z}^* \geq \frac{1}{\sqrt{L}} \norm{z}^*_{\SML}. \]

    Lastly, using this inequality in the strong convexity equation of $F$ in $\normSM{\cdot}^*$:
    \begin{align*}
        F(y) &\geq F(x) + \Abraces{\nabla F(x), y-x} + \frac{\sigma_{\SM}}{2} \paren{\norm{y-x}_{\SM}^*}^2 \\
        &\geq F(x) + \Abraces{\nabla F(x), y-x} + \frac{\sigma_{\SM}}{2L} \paren{\norm{y-x}_{\SML}^*}^2,
    \end{align*}
    from where we obtain $\frac{\sigma_{\SM}}{L} \leq \sigma_{\SML}$
\end{proof}

Theorem \ref{theo:rate_SGSL-CD} states that SGSL-CD converges (in expectation) at least as fast as SGS-CD. 
Algorithm SGSL-CD is also at least as fast as SL-CD by an argument analogous to Remark~\ref{rem:SGS_faster_than_SU}: for the same sequence of node activations, the setwise GSL rule achieves an equal or larger suboptimality reduction than the random coordinate sampling with the probabilities in  \eqref{eq:coord_samling_probs}.
We thus have the following result.

\begin{corollary} \label{corol:SGSL_fastest}
    In expectation, SGSL-CD converges equally fast or faster than both SGS-CD and SL-CD.
\end{corollary}

We remark that we could have compared the convergence rates of SGSL-CD and SL-CD following a procedure similar to the one used to compare SGS-CD and SU-CD, where we would define a norm using non-overlapping sets (in this case, accounting also for the coordinate-wise Lipschitz consants) as an intermediate step to compare the strong convexity constants $\sigma_{\Lcal}$ and $\sigma_{\SML}$.
We did this derivation and observed that just as it happened with $\sigma_{\SM}$ in eq. \eqref{eq:ineq_sigmaSM_sigmaA}, the lower bound on $\sigma_{\SML}$ is not tight and suggests that SL-CD could be faster than SGSL-CD, which as we argued above, is not true. 

Corollary \ref{corol:SL_faster_SU} states that SL-CD is faster than SU-CD, and Corollary \ref{corol:SGSL_fastest} states that SGSL-CD is the fastest of all algorithms analyzed here. 
However, these two methods depend on the knowledge of the coordinate-wise Lipschitz constants $L_\ell$ (see eq. \eqref{eq:coord_Lips_constants}). 
These constants are the global upper bounds on the diagonal entries of the Hessian $H = \nabla^2 F$, given by
\[ H_{\ell \ell}(\lambda) = \nabla^2 f_i^* (U_i^T \Lambda \lambda) + \nabla^2 f_j^* (U_j^T \Lambda \lambda), \;
\ell \equiv (i,j), \]
i.e. $H_{\ell \ell}(\lambda) \leq L_\ell \; \forall \lambda$.
We next describe a decentralized algorithm to estimate these values when they are not known. 
In Section \ref{sec:numerical_results} we show that the versions of SL-CD and SGSL-CD that use estimated constants, which we call SeL-CD and \mbox{SGSeL-CD} respectively, still perform remarkably well.

%
%

\subsection{Smoothness constants estimation} \label{ref:L_estimation}

In \cite{nesterov2012efficiency} the author proposed a method to estimate the value of the instantaneous Lipschitz constants $L_\ell(\lambda)$ when they are not known. 
By \textit{instantaneous} we mean the value of the Lipschitz constants at the current point $\lambda^k$, and not global values valid for any value of $\lambda$.

The procedure consists on finding, every time that variable $\lambda_\ell$ is going to be updated at iteration $k$, the \textit{lowest} value $L_\ell(\lambda^k)$ such that after applying update \eqref{eq:CD_update} with stepsize $\eta = 1/L_\ell(\lambda^k)$ it holds that 
$\nabla_\ell F(\lambda^k) \cdot \nabla_\ell F(\lambda^{k+1}) > 0$. 
In other words, the procedure searches for a Lipschitz constant (or equivalently, a stepsize) for which the update \eqref{eq:CD_update} does not \textit{overshoot}, making the gradient take a completely different direction. 

The procedure to estimate $L_\ell(\lambda^k)$ is shown in Algorithm \ref{alg:L_estimation}. 
In our numerical simulations, we denote SeL-CD and SGSeL-CD the versions of SL-CD and SGSL-CD that use estimated Lipschitz constants instead of the exact $L$ values.
SeL-CD is obtained by replacing line \ref{line:SU_update} in Algorithm \ref{alg:SU-CD} with Algorithm \ref{alg:L_estimation} and using the estimated values $L_\ell$ for the random sampling. 
SGSeL-CD is obtained by replacing line \ref{line:SGS_update} in Algorithm \ref{alg:SGS-CD} with Algorithm \ref{alg:L_estimation} and using the estimated values $L_\ell$ for the GSL neighbor choice. 

The choice of the initial value $\widehat{L}_\ell^0$ before entering the search loop is subject to a trade-off: if $\widehat{L}_\ell^0$ is too big, the loop will be exited after only one iteration but we risk being too conservative and making a much smaller step than we could. Conversely, if $\widehat{L}_\ell^0$ is too small, by repeated doubling we will eventually find the value $\widehat{L}_\ell$ that is closest to the true instantaneous smoothness $L_\ell(\lambda^k)$, but this may take many iterations inside the loop, which means many rounds of computation and communication for the nodes involved. 

How are these estimated values expected to perform with respect to the analytical ones? 
This depends heavily on the problem at hand. 
We can easily construct a case where the exact constants perform better than the estimated: assume that we are in the parallel distributed setting, where we perform primal optimization, and the function to optimize is $F(x) = x^T \diag(L_1,\ldots,L_d) x$. 
If $x^0 \neq \zeros$, then the algorithm using the analytic constants can converge in $d$ steps (one in each coordinate). 
This is actually what either SGS-CD or SGSL-CD would achieve.
However, using estimated constants will most likely exit the search loop finding values $\widehat{L}_\ell \neq L_\ell$, and thus will need more iterations.
Conversely, the analytical constants $L_\ell$ are \textit{global} quantities, and therefore, although they are valid in the complete optimization space, they might be very different to the real instantaneous Lipschitz constants for many values of $\lambda$ (or $x$ in the example above). 
In that case, we may get a much better approximation to the instantaneous value using the estimations, and therefore a faster convergence due to using a larger stepsize.     
In Section \ref{sec:exact_vs_estimated_L} we provide numerical tests where we observe both behaviors.   

%
\setlength{\textfloatsep}{10pt} 
\begin{algorithm}[t]
\caption{Online smoothness constant estimation}
\label{alg:L_estimation}
\begin{algorithmic}[1]
    \MYSTATE \textbf{Assumption:} Nodes $i$ and $j$ will update $\lambda_\ell, \ell \; {\equiv} \; (i,j)$ and they have already exchanged their $\nabla f_x^*(u_x^T A \lambda), x=i,j$. \\
    \textbf{Inputs:} Instantaneous smoothness starting value $\widehat{L}_\ell^0$ \\
    \textbf{Each node $x = i,j$ then runs:}
    \MYSTATE Compute $\nabla_\ell F(\lambda^k)$ with \eqref{eq:coord_gradient} using $\nabla f_x^*(u_x^T A \lambda), x=i,j$
    \MYSTATE Set $\widehat{L}_\ell \gets \widehat{L}_\ell^0$
    \MYSTATE \textbf{do ...}
        \MYSTATE \algindent Set $\widehat{L}_\ell \gets 2 \cdot \widehat{L}_\ell$ \label{line:double_L}
        \MYSTATE \algindent Compute 
        $\widehat{\lambda}_\ell = \lambda_\ell^k - (1/\widehat{L}_\ell) \cdot \nabla_\ell F(\lambda^k)$
        \MYSTATE \algindent Compute $\nabla f_x^*(u_x^T A \widehat{\lambda}_\ell)$ and send to neighbor
        \MYSTATE \algindent Compute $\nabla_\ell F(\widehat{\lambda})$ with \eqref{eq:coord_gradient} using $\nabla f_x^*(u_x^T A \widehat{\lambda}_\ell), x=i,j$
        \MYSTATE \algindent \textbf{... while} $\nabla_\ell F(\lambda^k) \cdot \nabla_\ell F(\widehat{\lambda}) \leq 0$
    \MYSTATE \textbf{end do-while}
\MYSTATE Set $L_\ell^{k+1} \gets 0.5 \cdot \widehat{L}_\ell$ and $\lambda_\ell^{k+1} \gets \widehat{\lambda}_\ell$
\end{algorithmic}
\end{algorithm}

%
%

\section{Additional considerations} \label{sec:additional_considerations}

%
%

\subsection{Application to parallel distributed optimization} \label{sec:parallel_distributed}

In the parallel distributed setup, the parameter vector is stored in a server accessible by multiple workers, each of which modifies some or all of the coordinates of the parameter. 
We assume that coordinates are updated by a single worker at each iteration and workers always access the most recent value of the parameter.

In this setting, if there are $E$ coordinates, $n$ workers, and we let each worker $i$ update a different set $\Scal_i$ of coordinates such that (i) the sets overlap, and (ii) each coordinate can be updated by exactly two workers, then all results presented previously (Theorems \ref{theo:rate_SU-CD}, \ref{theo:rate_SGS-CD}, \ref{theo:rate_SL-CD}, and \ref{theo:rate_SGSL-CD}) hold also for this setting. 
We remark these two conditions are \textit{not} necessary conditions to apply the setwise CD algorithms to the parallel distributed setting, but only  to directly apply the results of the theorems, which were derived for the decentralized setting. 
In fact, the family of setwise CD algorithms \textit{can always be applied} to the parallel distributed setting \textit{independently} of the degree of overlapping of the sets and the number of the coordinates modified by each worker.

We can then also easily construct a setting where SGS-CD is $N_{\max}$ times faster than SU-CD: 
let all sets have the same size $|\Scal_i| = N_{\max} \; \forall i$, exactly $(N_{\max}-1)$      coordinates in each set have $\nabla_m F(\lambda) = 0$, and only one $\ell$ have $\nabla_\ell F(\lambda) \neq 0$. 
In this case, on average \emph{only $\frac{1}{N_{\max}}$ times will SU-CD choose the coordinate that gives some improvement}, while SGS-CD will do it at all iterations.

Note that achieving the maximum speedup for this carefully crafted scenario requires that the gradients of all coordinates are independent, which is not verified in the decentralized optimization setting: according to eq. \eqref{eq:coord_gradient}, for a $\nabla_m F$ to be zero, it must hold that $\nabla f_i^* = \nabla f_j^*$ for $m \equiv (i,j)$. 
But unless this equality holds for \emph{all} $(i,j) \in \Ecal$ (i.e., unless the minimum has been attained), $\lambda$ will continue to change, and the $\nabla f_i^*$ will differ.
This prevents us from easily designing a scenario in the decentralized setting where SGS-CD attains the speedup upper bound with respect to SU-CD.
Nevertheless, in Figure \ref{fig:decen_and_distrib} we show examples where 
(i) \textit{the speedup increases linearly with $N_{\max}$} for the decentralized setting, and 
(ii) \textit{the speedup matches $N_{\max}$} for the parallel distributed setting.

%
%

%
%

\subsection{Dual-unfriendly functions and relation to Dual Ascent} \label{sec:dual_unfriendly}

The exposition that we have adopted up to this point may suggest that in order to run the setwise CD methods presented here, one should be able to compute the Fenchel conjugates $f_i^*$ for $i \in [n]$. 
Computing these functions may be tedious, and in some cases
simply impossible. 

However, we remark that the dual coordinate algorithm presented here is equivalent to the dual decomposition method (Section 2.2 in \cite{boyd2011distributed}) and therefore the gradients $\nabla f_i^*$ can be directly computed by minimizing the per-node Lagrangian (see also Proposition 11.3 in \cite{rockafellar2009variational})
\begin{equation} \label{eq:per-node_Lagrangian}
    \nabla f_i^*(u_i^T A \lambda) = 
    \arg\min_{\theta_i}  \Sbraces{ f_i(\theta_i) + \sum_{\ell \in \Scal_i} A_{i \ell} \lambda_\ell \theta_i}.
\end{equation}

Therefore, to apply the algorithms presented here we do \textit{not} need to be able to compute the Fenchel conjugates $f_i^*$, as long as we can solve \eqref{eq:per-node_Lagrangian} analytically or numerically to a high precision.
The latter is what we do in 
Section \ref{sec:numerical_results} for logistic regression.
We remark that a formal treatment of solving \eqref{eq:per-node_Lagrangian} inexactly would require an error propagation analysis similar to that done in Section 6 of \cite{uribe2020dual} for the synchronous setup.

%
%

\subsection{Case $d > 1$} \label{sec:d_larger_1}

To extend the proofs above for $d>1$, the block~arrays $\Lambda$ and  $U_i$ should be used instead of $A$ and $u_i$, and the selector matrices $T_i$ should be redefined in the same way (i.e., by making a Kronecker product with the identity). 
Then, all the operations that in the proofs above are applied \emph{per entry} (scalar coordinate) of the vector $\lambda$, should now be applied to \emph{the magnitude} of each vector coordinate $\lambda_\ell \in \Rbb^d$ of $\lambda \in \Rbb^{Ed}$. 
Also, since $\nabla_m F \in \Rbb^d$, in this case the GS rule becomes 
$\argmax_{m \in \Scal_i} \normtwo{\nabla_m F(\lambda)}^2$ (and the GSL rule is modified analogously).

%
%

\section{Numerical Results} \label{sec:numerical_results}

In this section, we test the algorithms proposed in numerical simulations and analyze their performance in a range of different scenarios. 
In all cases, we used \eqref{eq:per-node_Lagrangian} to compute the $\nabla f_i^*$ needed in \eqref{eq:coord_gradient}. 
For quadratic and linear least squares problems \eqref{eq:per-node_Lagrangian} has a closed-form expression, while for logistic regression we used the SciPy module.

%
%

%
%
\begin{figure}[tb]
\centering
\begin{subfigure}{0.51\linewidth}
    \centering \includegraphics[trim={1mm 1mm 1mm 1mm},clip,width=\linewidth]{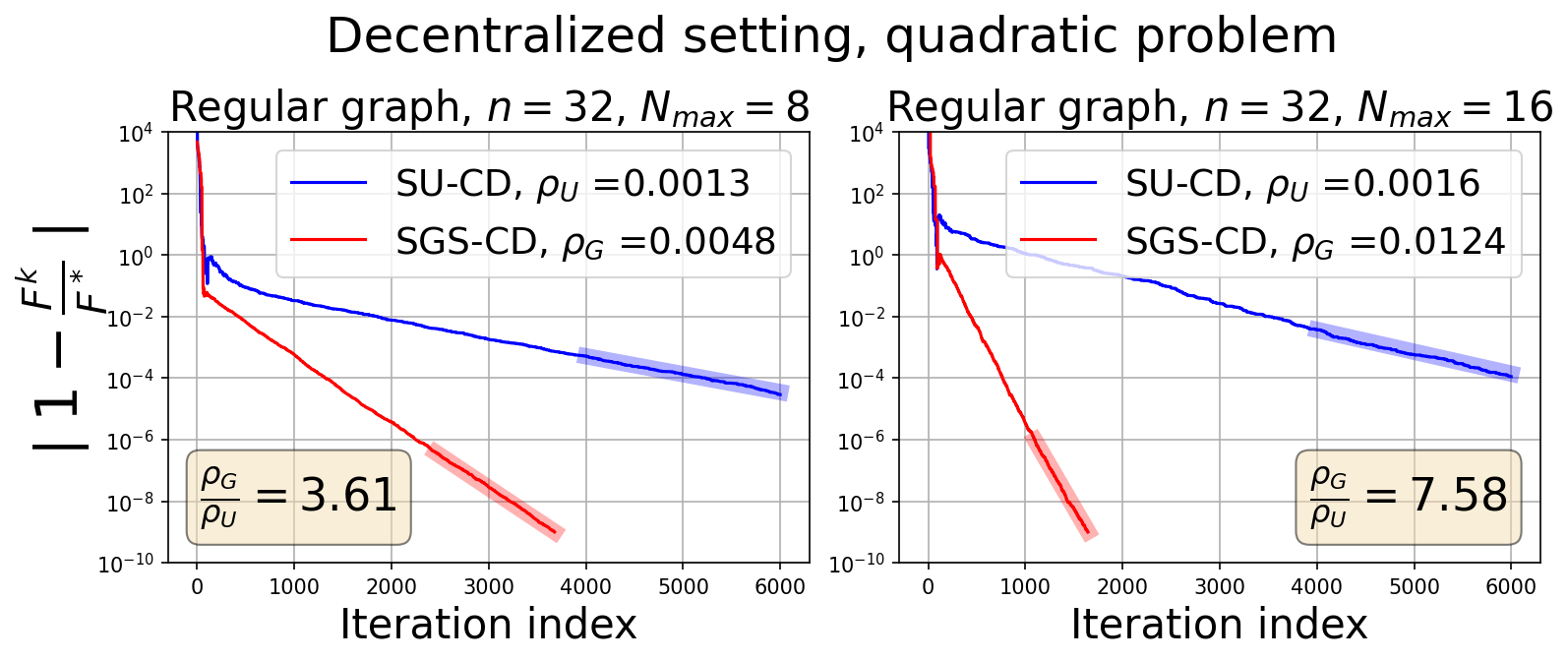}
\end{subfigure}
\begin{subfigure}{0.48\linewidth}
    \centering \includegraphics[trim={1mm 1mm 1mm 1mm},clip,width=\linewidth]{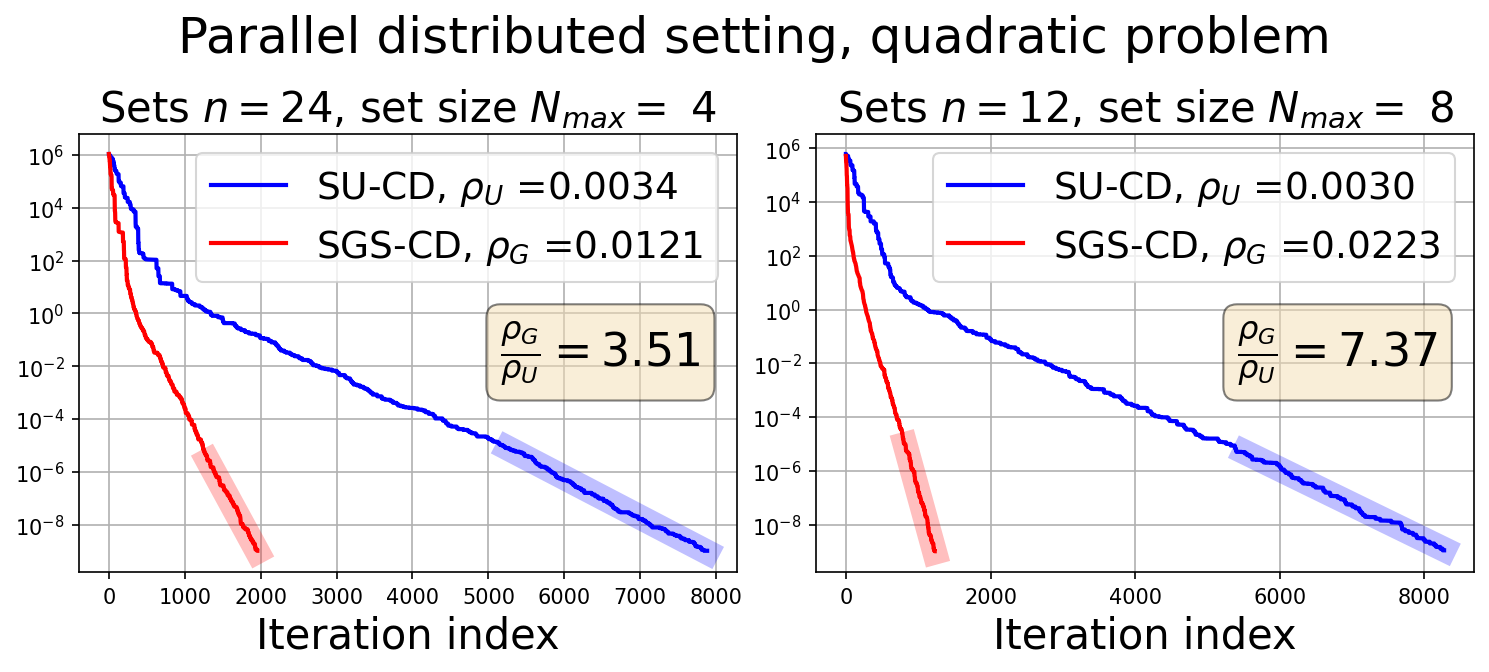}
\end{subfigure} 
\caption{Comparison of the convergence rates of SU-CD and SGS-CD for quadratic problems in two settings: decentralized optimization over a network (left plots), and parallel distributed computation with parameter server (right plots).}
\label{fig:decen_and_distrib}
\end{figure}

\subsection{SU-CD vs SGS-CD: speedup increase with $N_{\max}$}

Figure \ref{fig:decen_and_distrib} shows an example in the decentralized setting where the speedup of SGS-CD compared to SU-CD increases linearly with $N_{\max}$ (left plots), and an example in the parallel distributed setting where SGS-CD achieves the maximum speedup of $N_{\max}$ (right plots). 

For the decentralized setting, we created two regular graphs of $n=32$ nodes and degrees $N_{\max} = 8$ and 12, respectively. 
The local functions were $f_i(\theta) = \theta^T c I_d \theta$ with $d=5$, and the constant $c$ being much larger for one node than all others. This choice gave a few edges with smoothness constants much smaller than the rest, maximizing the chances to observe the advantages of SGS-CD versus SU-CD (see also the discussion in Section 4.1 of \cite{nutini2015coordinate}). 

For the parallel distributed setting, we created a problem that was separable per-coordinate, and we tried to recreate the conditions described in Section \ref{sec:parallel_distributed} to approximate the $N_{\max}$ gain. 
We chose $F(x) = x^T \diag(a_1,...,a_d) x$ with $d = 48$ and $a_i \sim \Ncal(10, 3) \; \forall i$. 
We then created $n$ sets of $N_{\max}$ coordinates such that each coordinate belonged to exactly two sets, and simulated two different distributions of the $d=48$ coordinates: one with $n=24$ sets of $N_{\max} = 4$ coordinates, and another with $n=12$ sets of $N_{\max} = 8$ coordinates. 
Following the reasoning in Section \ref{sec:parallel_distributed}, we set the initial value of $(N_{\max}-1)$ coordinates in each set to $x_m^0 = 1$ (close to the optimal value $x_m^*=0$), and the one remaining to $x_\ell^0 = 100$ (far away from $x_\ell^*=0$).

In all plots of Figure \ref{fig:decen_and_distrib} we used the portion of the curves highlighted with thicker lines to estimate the suboptimality reduction factor $(1-\rho)$, and called $\rho_U$ and $\rho_G$ the rates of SU-CD and SGS-CD, respectively. 
In all cases we see that $1 \leq \frac{\rho_G}{\rho_U} \leq N_{\max}$, as predicted by Theorem \ref{theo:rate_SGS-CD}. 
We additionally observe that this ratio increases approximately in the same proportion as $N_{\max}$ for the decentralized setting, and is approximately equal to $N_{\max}$ in the parallel distributed.

%
%

\subsection{Scaling with the number of nodes} \label{sec:scaling_with_num_nodes}

In addition to testing the effect of increasing $N_{\max}$, we repeated the experiment of the previous section increasing the number of nodes $n$ in the decentralized setting.
Figure \ref{fig:scaling_with_num_nodes} shows the results of this experiment for $N_{\max} = 8$ (we obtained analogous results for $N_{\max} = 16$). 
The figures confirm the conclusions of the previous section and illustrate how, although both algorithms get slower as the number of nodes increases (as could be expected), the speedup of SGS-CD versus SU-CD remains constant.

\begin{figure}[tb]
\centering
\includegraphics[width=0.6\linewidth]{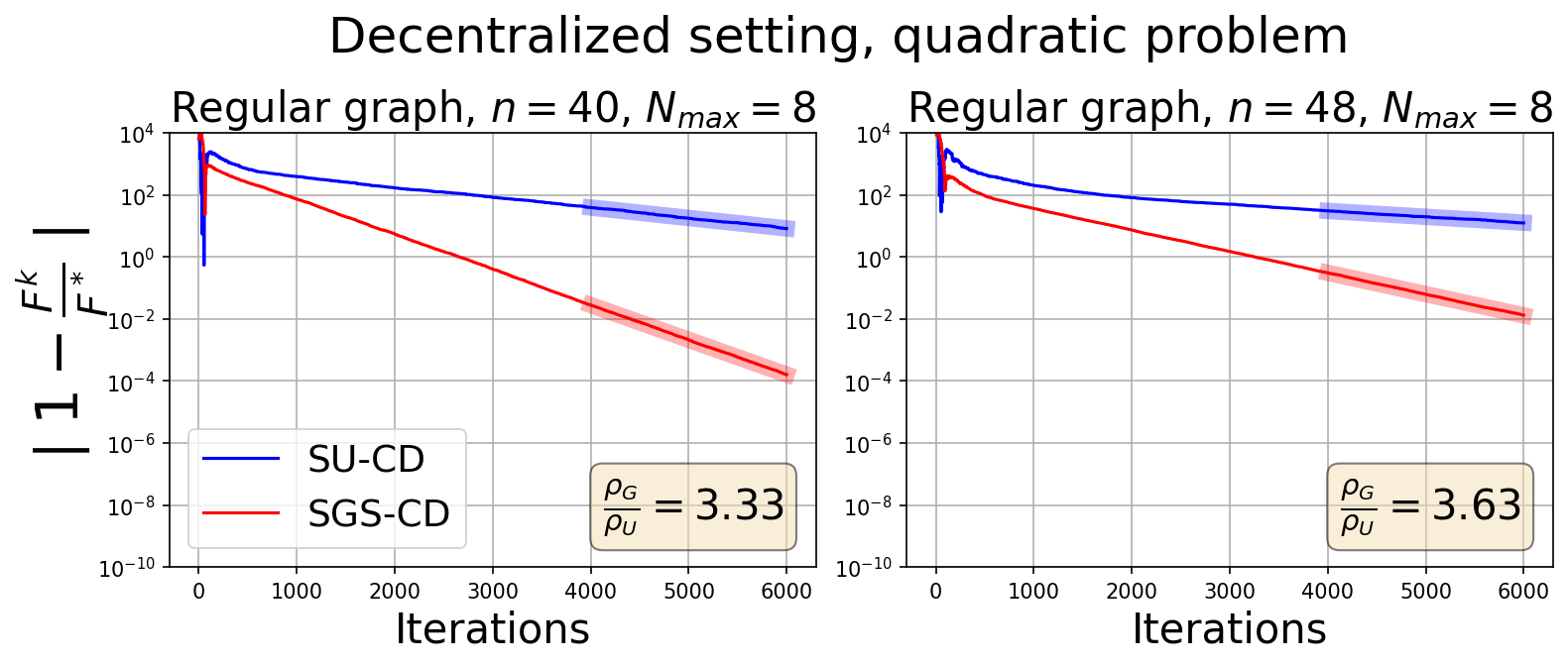}
\caption{Comparison of SU-CD and SGS-CD for increasing $n$.}
\label{fig:scaling_with_num_nodes}
\end{figure}

%
%

\subsection{Estimated vs exact Lipschitz constants} \label{sec:exact_vs_estimated_L}

As mentioned in Section \ref{ref:L_estimation}, whether the estimated instantaneous Lipschitz constants $\widehat{L}_i$ achieve a faster or slower per-iteration convergence of the exact global $L_i$ than the Lipschitz-informed algorithms depends heavily on the problem at hand. 
In this section, we compare SL-CD against SeL-CD in two different polynomial functions and show that the fastest algorithm is different in each case (Figure~\ref{fig:exact_vs_estim_L}).

We consider the parallel distributed setting of the last plot in Figure \ref{fig:decen_and_distrib}, where the parameter vector $x$ has 48 coordinates and 12 workers update 8 coordinates each, such that each coordinate is updated by exactly two workers. 
We compare the performance of SL-CD and SeL-CD in two functions: 
a quadratic 
$F(x) = a_1 x_1^2 + \ldots + a_d x_d^2 + 1$ and an order-four polynomial 
$F(x) = a_1 x_1^4 + \ldots + a_d x_d^4 + 1$. 
For the former, the global (and also the instantaneous) coordinate-wise Lipschitz constants are $\{L_i\} = \{2a_i\}$, while for the latter they are $\{L_i\} = \{12 a_i (\hat{x}_i)^2)\}$, where $\hat{x}_i$ is the maximum absolute value taken by the entry $x_i$ throughout the optimization.
The constants $a_i$ were set to random integers sampled at uniform in the interval [1,100]. 

Figure \ref{fig:exact_vs_estim_L} shows the performance of SL-CD and SeL-CD in the quadratic (left plot) and the order-four (right plot) problems. 
As expected, SL-CD converges faster than SeL-CD in the quadratic problem, where once a coordinate is selected for the first time it is set to its optimal value in that single iteration.
Note also that due to this behavior and the fact that the coordinate sampling probabilities of SL-CD are fixed (cf. eq. \eqref{eq:coord_samling_probs}), sampling a not-yet-optimized coordinate becomes more and more difficult as the iterations progress, which is why the convergence of SL-CD in the left plot of Figure \ref{fig:exact_vs_estim_L} shows a stairwise pattern where the length of the steps becomes larger with the iterations. 
The estimated constants $\widehat{L}_\ell$ approximate the optimal values $2a_i$ as well as possible, but cannot match them exactly and thus SeL-CD converges more slowly.

Conversely, in the order-four function SeL-CD converges faster than SL-CD, since in this case the estimated constants will always be closer to the true instantaneous coordinate-wise Lipschitz values than the global smoothness constants $L_\ell$. 
Note, however, that if the exact instantaneous values could be known at each iteration (which is the case of the quadratic function in the left plot of Figure \ref{fig:exact_vs_estim_L}), SL-CD using these values would always be faster than SeL-CD.

\begin{figure}[tb]
\centering
\includegraphics[width=0.6\linewidth]{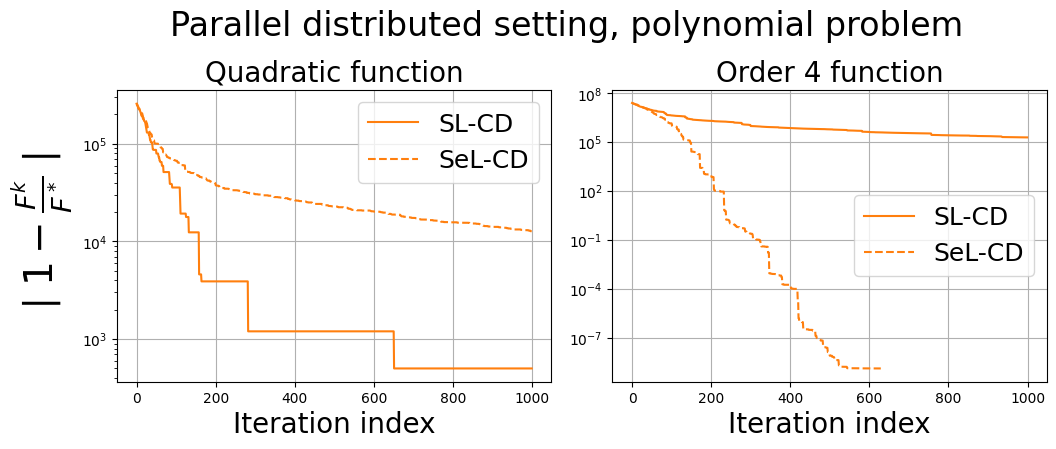}
\caption{Comparison of the performance of SL-CD (using global exact $L_i$ values) and SeL-CD in the parallel distributed setting. The experiment was designed to show that which of the two methods is faster depends on the problem considered.}
\label{fig:exact_vs_estim_L}
\end{figure}

%
%

\subsection{Number of iterations vs communication complexity} \label{sec:LLS}

Figure \ref{fig:LLS} shows the performance of all algorithms proposed for the linear least squares (LLS) problem
\[ f_i(\theta) = \frac{1}{M} \norm{X_i \theta - Y_i}_2^2, \; 
X_i \in \Rbb^{M \times d}, \; Y_i \in \Rbb^M, \]
in two random graphs of $n=32$ nodes and link probabilities of 0.1 (left plots) and 0.5 (right plots), respectively. 
The data was generated with the model of \cite{scaman2017optimal}, $d=5, M=30$, and the $Y$ values were additionally multiplied by the index of the corresponding node to have non-iid data between the nodes.
Here we show convergence both against number of iterations (top) and against number of vectors in $\Rbb^d$ transmitted through the network (bottom). 
Table~\ref{tab:comm_cost} shows the communication complexity of each algorithm in these terms.

\begin{table}[b]
\centering
\caption{Communication complexity of each algorithm: number of vectors in $\Rbb^d$ transmitted in one iteration for an arbitrary activated node $i$. 
Variable $I$ indicates the number of iterations inside the do-while loop in Algorithm \ref{alg:L_estimation}.}
\label{tab:comm_cost}
\begin{tabular}{||cc||cc||} \hline
SU-CD   & 2         & SGS-CD    & $N_i +1$ \\ \hline
SL-CD   & 2         & SGSL-CD   & $N_i +1$ \\ \hline
SeL-CD  & $2 + 2I$     & SGSeL-CD  & $N_i + 1 + 2I$ \\ \hline
\end{tabular}
\end{table}

In terms of the number of iterations, we confirm the conclusions of all our corollaries, namely
(i) SGS-CD converges faster than SU-CD (Corollary \ref{cor:SGS_faster_SU}), 
(ii) SL-CD converges faster than SU-CD (Corollary \ref{corol:SL_faster_SU}), and
(iii) SGSL-CD converges faster than both SL-CD and SGS-CD  (Corollary \ref{corol:SGSL_fastest}).
Whether the versions with estimated Lipschitz constants SeL-CD and SGSeL-CD are faster than their counterparts with exact Lipschitz knowledge SL-CD and SGSL-CD depends on the problem instance, as discussed in the previous section. 
Once again, the speedup of applying either the GS or the GSL rule increases radically as the graph becomes more connected. 


\begin{figure}[tb]
\centering
\includegraphics[trim={1mm 1mm 1mm 1mm},clip,width=0.6\linewidth]{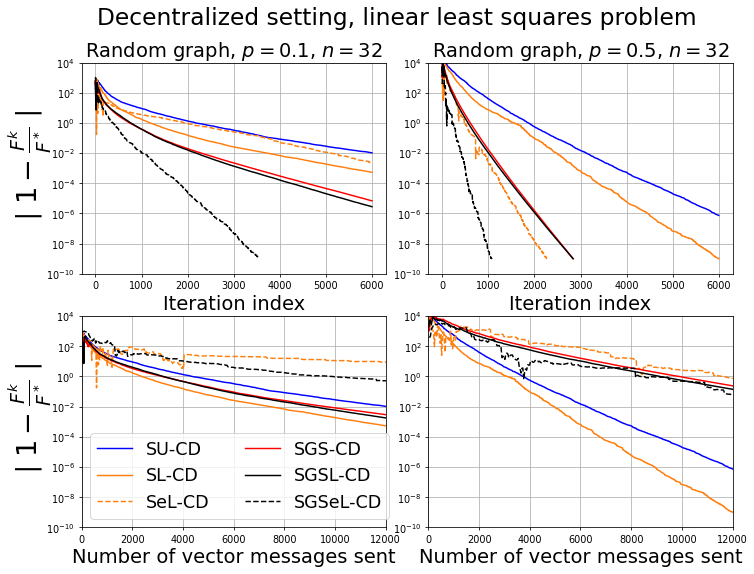}
\caption{Performance of the algorithms presented in an LLS problem and two random graphs with different numbers of edges (left and right columns). 
Top plots: convergence in terms of the number of iterations.
Bottom plots: convergence in terms of the number of vectors in $\Rbb^d$ transmitted.}
\label{fig:LLS}
\end{figure}


Although the plots against number of iterations allow us to confirm the predictions of our theory, they mask the important fact that the algorithms with highest rates have also larger communication complexity.
The plots in terms of the number of vectors transmitted illustrate better this trade-off. 
In the bottom plots we see that the suboptimality reduction per number of vectors transmitted can be comparable between the algorithms with randomized neighbor choice and those applying the GS(L) rule (left), and the former can even be more efficient as the network becomes denser (right). 
The greater overhead of the algorithms using estimated constants also becomes apparent in these plots. 
However, while the plots in Figure \ref{fig:LLS} support our theory and illustrate the communication overhead of the different algorithms, they do not reflect what we would expect to observe in a real system, where updates can happen in parallel. 
We analyze this in the next section.

%
%

\subsection{Asynchronism and communication overhead} \label{sec:asynchrony_realistic}

In this section, we analyze how the characteristics of a realistic setting, with asynchronous node activations and communication overhead, can affect the relative performance of the algorithms. 
For clarity, we limited our analysis to SU-CD and SGS-CD;
Table \ref{tab:comm_cost} and Figure \ref{fig:LLS} allow to extrapolate the conclusions here to the rest of the algorithms presented. 

In Figure \ref{fig:asynchrony_realistic} we repeat the simulations of Figure \ref{fig:decen_and_distrib} for the decentralized setting and graph degree $N_{\max} = 16$, now plotting against time (where units are arbitrary). 
We consider that nodes activate following an exponential distribution with average activation rate $\kappa$ and that communication takes time $\tau$, during which the transmitting nodes are blocked and cannot either start or participate in other updates. 
We consider $\kappa \in \{10, 5\}$ and $\tau \in \{0, 1\}$.
We also consider that the activation rates can be equal to $\kappa$ for all nodes (top two plots), or can have a skewed distribution with expected value $\kappa$ (bottom plot). 
This skewed distribution (we chose a Zipf with parameter 2) intends to model the presence of slow nodes. 

The top plot in Figure \ref{fig:asynchrony_realistic} shows that if communication time is negligible ($\tau=0$), we recover the convergence curves of Figure \ref{fig:decen_and_distrib}, independently of the parameter $\kappa$.
There is only a shrinkage of the time axis as we go from $\kappa=10$ to 5 due to updates occurring faster.

The differences show up when communication is not instant anymore ($\tau=1$, middle plot). 
Here the gains of SGS-CD versus SU-CD are smaller than in the top plot, since now in \mbox{SGS-CD} nodes have to wait for \emph{all} their neighbors to be available. 
When we increase the activation rates from $\kappa=10$ to 5, the gap between the algorithms reduces further, since SU-CD can complete updates more quickly but SGS-CD suffers of greater overhead.

The bottom plot shows that having skewed node activations exacerbates this effect: most nodes have faster rates than $\kappa$ and this speeds up the convergence of SU-CD, but \mbox{SGS-CD} is constrained by its overhead. 
We remark, however, that both algorithms converge faster than a synchronous algorithm would, since in the synchronous setting the updates would be completed at the rate of the slowest node. 

\begin{figure}[tb]
\centering
\begin{subfigure}{0.6\linewidth}
    \centering \includegraphics[trim={1mm 1mm 1mm 1mm},clip,width=\linewidth]{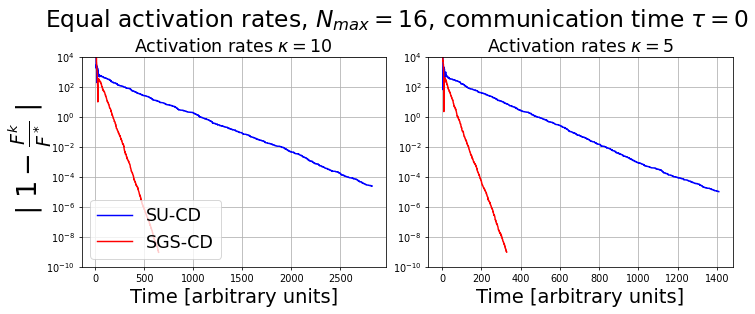}
\end{subfigure}
    \\ [0.3em]
\begin{subfigure}{0.6\linewidth}
    \centering \includegraphics[trim={1mm 1mm 1mm 1mm},clip,width=\linewidth]{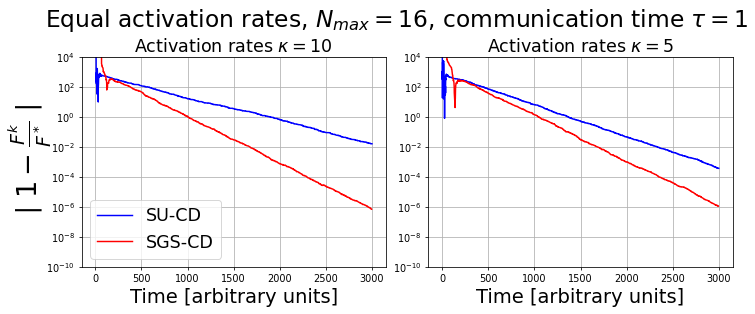}
\end{subfigure}
    \\ [0.3em]
\begin{subfigure}{0.6\linewidth}
    \centering \includegraphics[trim={1mm 1mm 1mm 1mm},clip,width=\linewidth]{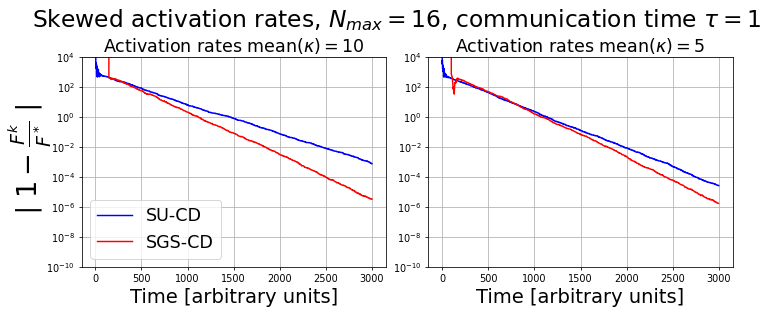}
\end{subfigure}
\caption{Comparison of SU-CD vs. SGS-CD in a realistic setting where (i) nodes activation intervals are exponentially distributed with average parameter $\kappa$, and (ii) communication takes time $\tau$. We consider equal and skewed activation rates.}
\label{fig:asynchrony_realistic}
\end{figure}

%
%

\subsection{Comparison against token algorithms} \label{sec:comparison_with_token}

In this section we compare the performance of SU-CD and SL-CD against that of token algorithms \cite{mao2020walkman, hendrikx2023principled}, which have the same edge asynchronicity model of the algorithms considered in this paper and share similar computation and communication complexities. 

Figure \ref{fig:comparison_token_algorithms} shows the performance of the algorithms in terms of number of iterations (left) and in the realistic setting described in Section \ref{sec:asynchrony_realistic} (right) with parameters $\kappa=10$ and $\tau=1$ for the quadratic problem in a regular graph of $n=32$ and degree $N_{\max}=8$. 
For Walkman \cite{mao2020walkman}, we set their stepsize $\beta$ to the maximum allowed by their theory. 
For the multi-token algorithms we simulate $K \in \{1, 5, 10, 15\}$ tokens and perform one computation followed by one communication (among the many variants that the authors propose in their paper). 
Remark that since Walkman is a primal algorithm, here we plot primal suboptimality.

\begin{figure}[tb]
\centering
\includegraphics[trim={0cm 0cm 12.5cm 1.2cm}, clip, width=0.6\linewidth]{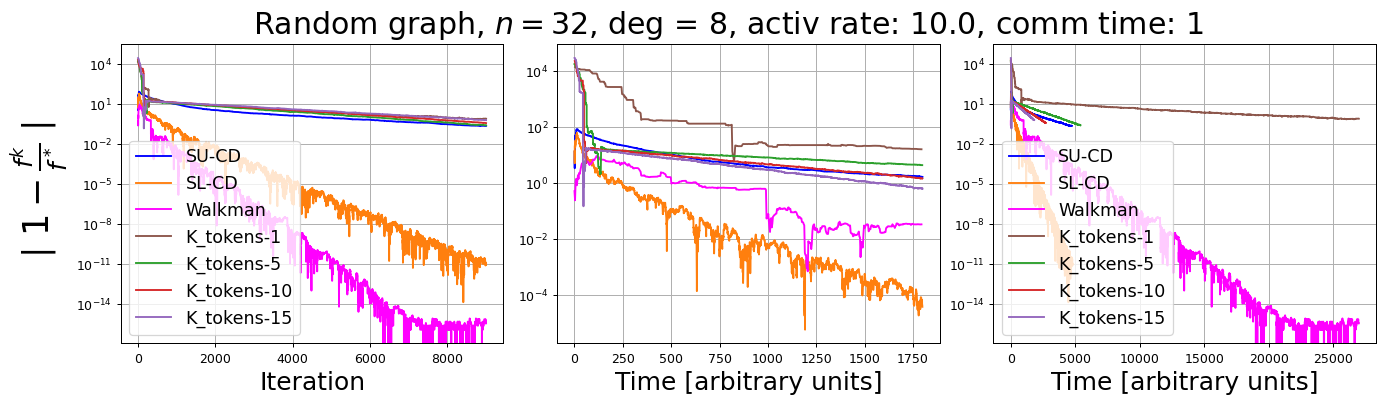}
\caption{SU-CD and SL-CD against token algorithms.}
\label{fig:comparison_token_algorithms}
\end{figure}

Walkman is the algorithm that makes the laregest progress per iteration (left). 
Note that in terms of iterations the multi-token algorithms perform equally, since it does not matter what token makes the update.
The advantages of the asynchronicity and multiple updates of our algorithms and the multi-token are seen in the right plot, for the realistic setting, where many updates can happen in parallel. 
Here SL-CD is the fastest of all thanks to allowing many updates in parallel, in contrast to Walkman whose updates are sequential. 
While SU-CD also allows parallel updates, its convergence is slower due to the use of the same stepsize for all dual coordinates. 
Still, it performs similarly to the multi-token algorithm with $K=10$.
We remark that the authors of \cite{hendrikx2023principled} provided variants of their algorithms with local updates and acceleration that would make the convergence of multi-token faster, but would not be a fair comparison in terms of communication and computation.

%
%

\subsection{A dual-unfriendly problem with no $L$ knowledge} \label{sec:LR}

Figure \ref{fig:LR} shows the convergence of SeL-CD and SGSeL-CD in the logistic regression problem 
\[ f_i(\theta) = \frac{1}{M} \sum_{j=1}^M \log \big( 1 + \exp (-[Y_i]_j \cdot ([X_i]_j)^T \theta_i ) \big) + c \norm{\theta_i}_2^2 \]

where $[Y_i]_j$ and $([X_i]_j)^T$ are the $j$-th component and the \mbox{$j$-th} row of arrays $Y_i \in \Rbb^M$ and $X_i \in \Rbb^{M \times d}$, respectively. 
We ran the simulation for the same graphs and parameters used for the experiments in Section \ref{sec:LLS}.
In this case, we cannot compute analytically the optimal value of \eqref{eq:per-node_Lagrangian}, so we did the optimization in \eqref{eq:per-node_Lagrangian} using the SciPy module. 
For the same reason, we do not know the true coordinate Lipschitz values $L_\ell$, so we test only the algorithms using estimated constants. 

As in the previous examples, both algorithms converge linearly, and SGSeL-CD is faster than SeL-CD. 
We may remark, however, that the gap between the two algorithms does not increase with the graph connectivity, as observed between \mbox{SL-CD} and SGSL-CD in the top plots of Figure \ref{fig:LLS}. 
Indeed, we observed that when the $L_\ell$ are estimated, the gap between SGSeL-CD and SeL-CD may or may not increase with the connectivity of the graph. 
We attribute this effect to the fact that the performance of the algorithms depends very much on how close to optimal the estimated Lipschitz constants are, and therefore, some instances that allow for a better fit of the true constants using Algorithm \ref{alg:L_estimation} have advantage over others whose true constants cannot be well approximated.
In the next section we discuss how to improve the estimation of Algorithm \ref{alg:L_estimation}, and the associated costs of this improvement.

\begin{figure}[h]
\centering
\includegraphics[width=0.6\linewidth]{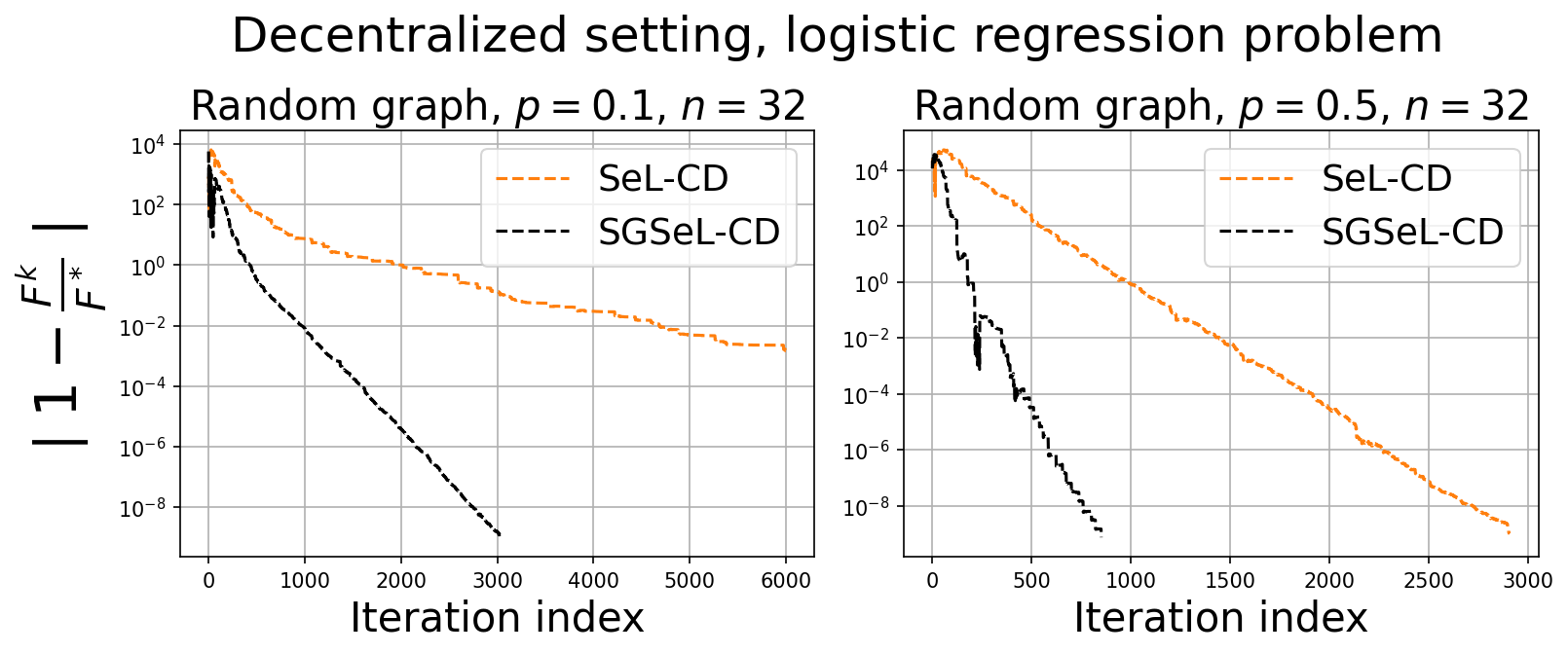}
\caption{Convergence of SeL-CD and SGSeL-CD in a logistic regression problem.}
\label{fig:LR}
\end{figure}

%
%

\section{Discussion and Conclusion} \label{sec:conclusion}

We have presented the class of \emph{setwise CD} optimization algorithms, where in a multi-agent system workers are allowed to modify only a subset of the total number of coordinates at each iteration. 
These algorithms are suitable for (dual) asynchronous decentralized optimization and (primal) distributed parallel optimization. 

We studied the convergence of a number of setwise CD variants: random uniform and Gauss-Southwell setwise coordinate selection (SU-CD and SGS-CD), and their Lipschitz-informed versions (SL-CD and SGSL-CD).
We showed linear convergence for all variants for smooth and strongly convex functions $f_i$, which required developing a new methodology that extends previous results on CD methods. 

In particular, we proved that in expectation, when convergence is measured \textit{in terms of number of iterations}, both SGS-CD and SL-CD are faster than SU-CD, and SGSL-CD is the fastest of them all. 
However, the algorithms applying the GS(L) rule entail higher communication, which can potentially lower the rate at which iterations are completed and give milder speedups in a real system. 
The same holds for the algorithms using estimated Lipschitz constants, which require transmitting even more vectors per iteration.
Nevertheless, our simulations for a realistic setting with asynchronous activations show that the gains of the GS rule can still be significant with respect to the uniform neighbor choice.

An interesting question is whether the speedup derived for SGS-CD with respect to SU-CD can be obtained with more lenient assumptions than smoothness and strong convexity of the local objectives. 
Since we solve the problem in the dual, the smoothness assumption may be removed (leading to a just-convex dual) using the methodology of \cite{stich2017approximate}. 
Removing the strong convexity assumption would lead to a non-smooth dual, and we are not aware of any results that provide the speedup of the GS rule without the smoothness assumption.
Filling this gap in the literature would constitute an important contribution to the current knowledge of CD algorithms and may be explored in future work.

%
%

\section*{Acknowledgments}

M.C. acknowledges funding from the European Union's H2020 MonB5G (grant no. 871780) project.
P.M. is also a member of Archimedes/Athena RC—NKUA, and was partially supported by the PEPR IA FOUNDRY project (ANR-23-PEIA-0003) of the French National Research Agency (ANR), and project MIS 5154714 of the National Recovery and Resilience Plan Greece 2.0 funded by the European Union under the NextGenerationEU Program.
T.S. acknowledges funding from the Hellenic Foundation for Research \& Innovation (HFRI) project 8017: ``AI4RecNets: Artificial Intelligence
(AI) Driven Co-design of Recommendation and Networking Algorithms".

\bibliographystyle{ieeetr}
\bibliography{bibliography}

\end{document}